\documentclass[11pt,a4paper]{amsart}
\usepackage{amsmath,amsfonts,amsthm,amssymb}
\usepackage[T1]{fontenc}
\usepackage{color}
\usepackage{epsfig}
\usepackage{graphicx}
\usepackage{cite,url}
\usepackage{hyperref}
\usepackage{xspace}
\usepackage{enumitem}
\usepackage[left=1in,right=1in,top=1in,bottom=1in,foot=0.5in]{geometry}
\usepackage{thm-restate}
\usepackage{soul}

\usepackage{amsthm}
\newtheorem{theorem}{Theorem}
\newtheorem{proposition}{Proposition}
\newtheorem{lemma}{Lemma}

\theoremstyle{definition}
\newtheorem{definition}{Definition}

\newtheorem{example}{Example}

\numberwithin{equation}{section}

\DeclareMathOperator{\supp}{supp}
\DeclareMathOperator{\sign}{sign}

\DeclareMathOperator{\Baryc}{Baryc}
\DeclareMathOperator{\Cocirc}{Cocirc}
\DeclareMathOperator{\Circ}{Circ}

\DeclareMathOperator{\Sign}{Sign}
\DeclareMathOperator{\Min}{Min}

\newcommand{\covectors}{\ensuremath{L}}
\newcommand{\X}{\ensuremath{X}}
\newcommand{\HH}{\ensuremath{H}}

\newcommand{\MM}{\ensuremath{M}}

\newcommand{\RR}{\ensuremath{R}}
\newcommand{\FF}{\ensuremath{F}}
\newcommand{\UU}{\ensuremath{U}}
\newcommand{\OO}{\ensuremath{{O}}}
\newcommand{\hcovectors}{J}

\begin{document}

\thispagestyle{empty}

\title{Geometry of ample/lopsided sets}

\author{Hans--J\"urgen Bandelt}
\address{Fachbereich Mathematik, Universit\"at   Hamburg, Bundesstr. 55, D-20146 Hamburg, Germany}
\email{bandelt@math.uni-hamburg.de}

\author{Victor Chepoi}
\address{Laboratoire d'Informatique et des Syst\`emes, Universit\'e d'Aix-Marseille, Facult\'e des Sciences de Luminy, F-13288 Marseille Cedex 9, France}
\email{victor.chepoi@lis-lab.fr}

\author{Andreas Dress}

\author{Jack Koolen}
\address{University of Science and Technology of China, China}
\email{koolen@ustc.edu.cn}

\begin{abstract} Lopsided sets were introduced by Jim Lawrence in 1983 when he studied the
subsets $\covectors$ of   $\{-1,+1\}^E$ that encode the intersection
pattern of a convex set $K$  with the orthants of  ${\mathbb R}^E$.
He identified a strong combinatorial condition  on
$\covectors$ (that he dubbed \emph{lopsidedness}) that is necessary (yet, as he showed, not sufficient) for the existence of such a convex set $K$, gave a number of equivalent conditions, and noted its close connections with various other topics studied in combinatorics. Lopsided sets have been independently rediscovered by several other authors, in particular by Andreas Dress in 1995, who called them \emph{ample} sets. Dress defined  ample sets as the set families satisfying equality in a combinatorial inequality, which holds for all set families.

Ample/lopsided sets can be regarded as the isometric subgraphs of  hypercubes for which the intersections with all subhypercubes yield isometric subgraphs.  Alternatively, they correspond to set families $\covectors$ such that when a subset $A$ of the ground set $E$ is shattered by $\covectors$, then $A$ is strongly shattered by $\covectors$, i.e., some subhypercube defined by $A$ is fully included in $\covectors$.

 In a previous article we characterized ample sets in various combinatorial and graph-theoretical ways. In this paper we study geometric realizations of ample sets as cubihedra (cube complexes), which yields several new characterizations. 
One such characterization establishes that the cubihedra of ample sets endowed with the intrinsic $\ell_1$-metric are exactly the isometric subspaces of $\ell_1$-spaces (which we call, weakly convex sets). We also view the barycenter maps of faces of cubihedra of ample sets as collections of $\{ \pm 1, 0\}$-sign vectors and, in analogy with the characterization  of oriented matroids by the covectors and the cocircuits. Moreover, we characterize the  collections of $\{ \pm 1, 0\}$-sign vectors corresponding to barycenter maps of all faces and all maximal faces of an ample set.  Furthermore, we show that any ample set $\covectors\subseteq \{ -1,+1\}^E$  is realizable as the intersection pattern of a weakly convex set $K$ with the orthants of ${\mathbb R}^E$.
All this testifies that the concept of ample sets is quite natural in the context of cube complexes.
\end{abstract}

\maketitle

\section{Introduction}

This paper is the follow-up of \cite{BaChDrKo}, in which we presented  a list of  combinatorial, recursive,  and graph-theoretical  characterizations of
\emph{lopsided sets} first introduced and investigated by Lawrence \cite{La} and rediscovered independently by Bollob\`as and Ratcliffe \cite{BoRa}, Dress \cite{Dr}, and Wiedemann \cite{Wi}, who dubbed them \emph{extremal}, \emph{ample}, and \emph{simple sets}, respectively.
In the present paper, we call such sets \emph{ample} (adopting the name from \cite{Dr}, which we find the most appropriate since it perfectly reflects their combinatorial quintessence) and we provide  several geometric
characterizations of ample sets, each emphasizing one or
another feature of ampleness and its relationships with some fundamental
concepts from the geometry of
$\ell_1$-spaces. We will repeatedly refer to the first part \cite{BaChDrKo}
for definitions, properties and characterizations of ample  sets.

\subsection*{Ampleness} Throughout this paper, $E$  denotes a finite set with
$n:=\#E$ elements,   $\{ \pm 1\}^E$ is  the set  of all maps from  $E$
into $\{ \pm 1\}=\{ -1,+1\}$, $\{ \pm 1, 0\}^E$ is  the set  of all maps from  $E$
into  $\{ \pm 1, 0\}=\{ -1,0,+1\}$. Thus  ${\mathbb R}^E$ is the vector space consisting of all maps from $E$ into $\mathbb R$. The \emph{solid hypercube}
$\HH(E)=[-1,+1]^E\subset {\mathbb R}^E$ consists of all maps from $E$ into
the interval $[-1,+1]$. We consider $\HH(E)$ as a cube complex, i.e., as a convex polyhedron together with its face structure. 
The 0-skeleton of $\HH(E)$ is $\{ \pm 1\}^E$, its 1-skeleton ${\HH}^{(1)}(E)$ is the \emph{graphic hypercube}, and $\{ \pm 1, 0\}^E$ is the set of sign maps of the barycenters of faces of $\HH(E)$. For any two maps $r',r''\in {\mathbb R}^E$, we denote by
$d(r',r'')$ their $\ell_1$-distance $||r',r''||_1=\sum_{e\in E} |r'(e)-r''(e)|$. Two vertices $s',s''\in \{ \pm 1\}^E$ define an edge (of length 2) of  ${\HH}^{(1)}(E)$
if and only if $d(s',s'')=2.$

In most  of the paper, we will consider subsets ${\covectors}$ of $\{ \pm 1\}^E$. Alternatively, such a set $\covectors$ can be viewed as a collection of subsets of $E$.  The set-theoretic complement of $\covectors$ is written as
${\covectors}^*:$
$${\covectors}^*:=\{ \pm 1\}^E-{\covectors}.$$
We denote by $G({\covectors})$ 
the subgraph of the graphic hypercube ${\HH}^{(1)}(E)$ induced by  $\covectors$. The set $\covectors$ is
called {\it connected}
if $G({\covectors})$ is connected, and it is called {\it isometric} if every
pair of vertices $s',s''$ of $\covectors$ can be connected in $G({\covectors})$
by a path of length $d(s',s'').$

Given any subset $A$ of $E,$ one can always associate two subsets
${\covectors}_A$ and ${\covectors}^A$ of $\{ \pm 1\}^{E-A}$ with an arbitrary set
${\covectors}\subseteq \{ \pm 1\}^{E}$ of sign maps:
\begin{align*}
{\covectors}_A:= \{ t\in \{ \pm 1\}^{E-A}: \mbox{ some extension } s\in \{ \pm 1\}^E \mbox{ of } t
\mbox{ belongs to } {\covectors} \},\\
{\covectors}^A:= \{ t\in \{ \pm 1\}^{E-A}: \mbox{ every extension } s\in \{ \pm 1\}^E \mbox{ of } t
\mbox{ belongs to } {\covectors} \}.
\end{align*}
${\covectors}_A=\{ s|_{E-A}: ~s\in {\covectors}\}$ encodes the projection of $\covectors$ onto $\{\pm 1\}^{({E-A})}.$
In contrast to ${\covectors}_A$, the smaller set ${\covectors}^A$ requires the existence of a full fiber
isomorphic to $\{\pm 1\}^{A}$  within $\covectors$ rather than just one point from $\covectors$.
The two operators ${\covectors}_A$ and ${\covectors}^A$ suggest two ways
to derive an abstract simplicial complex from $\covectors$:
\begin{align*}
\overline{\X}({\covectors}):=\{ A\subseteq E: ~{\covectors}_{E-A}=\{ \pm 1\}^A\},\\
\underline{{\X}}({\covectors}):=\{ A\subseteq E: ~{\covectors}^A\ne \varnothing\}.~~~~~~~~~~~
\end{align*}
Simplices of $\overline{{\X}}({\covectors})$ are  usually called  sets \emph{shattered} by $\covectors$  and  simplices of $\underline{{\X}}({\covectors})$ are called sets \emph{strongly shattered} by $\covectors$ \cite{ChChMoWa,MoWa}. The largest size of a simplex of $\overline{{\X}}({\covectors})$ is called the
\emph{Vapnik-Chervonenkis dimension} (\emph{VC-dimension} for short) of $\covectors$ (viewed as a set family). 

\begin{example}\label{example1}
Let $E=\{ 1,2,3\}$ and consider the subset $\covectors$ of $\{ \pm 1\}^{E}$ that consists
of all sign maps except the two constant ones. Then $\covectors$ encodes an
isometric 6-cycle in $\HH^{(1)}(E).$ For every singleton $A$ the projection of $\covectors$
onto $\{ \pm 1\}^{E-A}$ is surjective, but $\covectors$ does not include a full fiber isomorphic to this 2-cube.
Therefore $\overline{{\X}}({\covectors})$ comprises all proper subsets of $\{ 1,2,3\},$ whereas
$\underline{{\X}}({\covectors})$ consists
of the empty set and the three singletons.
\end{example}

In \cite{Dr} and \cite{BaChDrKo}  it is shown that
$$\#\underline{{\X}}({\covectors})\le \#{\covectors}\le \#\overline{{\X}}({\covectors})$$
holds for any $\covectors\subseteq \{ \pm 1\}^E$ (the inequality $\#{\covectors}\le \#\overline{{\X}}({\covectors})$ was proved before by Pajor \cite{Pa}). 
A set $\covectors$ is called {\it ample} if the equality
$\#{\covectors}=\#\overline{{\X}}({\covectors})$ holds.  Ampleness
turned out to be preserved when passing to the complementary set ${\covectors}^*$ and to
the sets ${\covectors}^A, {\covectors}_A$, and to imply connectedness (and, even more,
isometricity) of $L$. 
It followed that ${\covectors}_A$ and ${\covectors}^A$ had to be connected (isometric) subsets of $\{\pm 1\}^{E-A}$
for every ample subset ${\covectors}$ of $\{ \pm 1\}^E$. Conversely,
connectivity (or isometricity) of ${\covectors}^A$ for all $A\subseteq
E$ turned out to imply ampleness, suggesting to call such subsets
{\it superconnected} or {\it superisometric}. Further
investigation finally resulted in recognizing that our ample sets
coincided exactly with Lawrence's lopsided sets and that an
amazingly rich and multi--facetted theory regarding such subsets
of $\{ \pm 1\}^E$ could be developed. Here is a list of the most  remarkable
properties of ample sets established in \cite{BaChDrKo}, each of which could
be used to define them:

\medskip\noindent
{\it {\rm(1)} Superisometry:} ${\covectors}^A$ is isometric for all $A\subseteq
E,$

\medskip\noindent
{\it {\rm(2)} Superconnectivity:} ${\covectors}^A$ is connected for all
$A\subseteq E,$

\medskip\noindent
{\it {\rm(3)} Isometric recursivity:} ${\covectors}$ is isometric, and both
${\covectors}_e$ and ${\covectors}^e$ are ample for some $e\in E$,

\medskip\noindent
{\it {\rm(4)} Connected recursivity:} ${\covectors}$ is connected,
and ${\covectors}^e$ is ample for
every $e\in E,$

\medskip\noindent
{\it {\rm(5)} Commutativity:}   $({\covectors}^A)_B=({\covectors}_B)^A$ holds for
any disjoint  $A,B\subseteq E,$

\medskip\noindent
{\it {\rm(6)} Ampleness:}  $\#{\covectors}=\#\overline{\X}({\covectors}),$

\medskip\noindent
{\it {\rm(7)} Sparseness}: $\# {\covectors}=\# \underline{{\X}}({\covectors}),$

\medskip\noindent
{\it {\rm(8)} Lopsidedness}: for all $A,B\subseteq E$ with $A\cap B=\varnothing$ and $A\cup B=E$, either $A\in  \underline{{\X}}({\covectors})$ or $B\in  \underline{{\X}}({\covectors}^*)$.

\medskip\noindent
{\it {\rm(9)} Hereditary Euler characteristic 1:} for every face $F$ of $\HH(E)$ intersecting $\covectors$,
$$\sum_{i\ge 0} (-1)^if_i({\covectors}\cap F)=1,$$
where $f_i({\covectors}\cap F)$ counts the number of graphic $i$-cubes contained in the
intersection of $\covectors$ with $F$.

\medskip
The second-last characterization was the original definition of lopsidedness by Lawrence \cite{La}. The last characterization was first proved by Wiedemann  \cite{Wi}. We also recall the following nice characterization of ample sets due to Lawrence \cite{La}:

\medskip\noindent
{\it {\rm(10)} Total asymmetry:} If the intersection of ${\covectors}$ with a face $F$ of $\HH(E)$ is closed by taking antipodes, then either $F\cap \covectors$ is empty or coincide with $F\cap \{ \pm 1\}^E$.

\medskip
Lawrence \cite{La} presented several examples of ample sets. In particular, he proved that the sign vectors of the closed orthants of $\mathbb R^E$ intersecting a convex set $K$ is ample (he also presented an example of an ample set not realizable in this way). Additionally, in \cite{BaChDrKo}, we proved that vertex-sets of median graphs (CAT(0) cube complexes) are ample, and, more generally, that convex geometries and their bouquets  (which we called conditional antimatroids) are ample.

\subsection*{Our results}
Every subset ${\covectors}$ of $\{ \pm 1\}^X$ gives rise to a (not necessarily connected)
{\it cube complex} comprising all faces of the hypercube ${\HH}(E)$ all of whose vertices
belong to ${\covectors};$ cf. \cite{VdV}. This (compact) cubical polyhedron
(a {\it cubihedron} for short) will be denoted by $|{\covectors}|$ and
called the  {\it geometric realization} of $\covectors$. The vertices of $|{\covectors}|$
are exactly the elements in ${\covectors}$.  
The 1-skeleton  of this cubihedron is the graph $G({\covectors})$ defined above. 
If
$\covectors$ is connected, then $|{\covectors}\vert$ is connected
as well and therefore can be endowed with an intrinsic
$\ell_1$-metric $d_{|{\covectors}|}$. The resulting metric space
$(\vert {\covectors}\vert,d_{|{\covectors}|})$ is complete and
geodesic but is not
necessarily \emph{path-$\ell_1$-isometric}  (in the sense that any $r',r''\in \vert {\covectors}\vert$
can be connected in $\vert {\covectors}\vert$ by an $\ell_1$-geodesic).
For example, if $\covectors$ is defined as in Example \ref{example1},
then $|{\covectors}|$ is a solid 6-cycle of ${\mathbb
R}^3.$ The $\ell_1$-distance between the midpoints of two opposite
sides of this cycle is 4, while the intrinsic $\ell_1$-distance is 6.

In this paper, we will establish
that  path-$\ell_1$-isometricity  of 
the associated cube complex $\vert {\covectors}\vert$
is yet another characteristic feature of ampleness,
thus demonstrating that ample sets constitute a fundamental
domain for $\ell_1$-geometry:

\medskip\noindent
{\it {\rm(11)} Path-$\ell_1$-isometricity:}  $(\vert {\covectors}\vert,d_{|{\covectors}|})$ is  a
path-$\ell_1$-isometric subspace of the $\ell_1$-space $({\mathbb
R}^E,d)$.

\medskip
Since  the metric space $(\vert {\covectors}\vert,d_{|{\covectors}|})$  is complete, by Menger's theorem  \cite{Me},
the path-$\ell_1$-isometricity of $|\covectors|$   is equivalent to its Menger convexity, which in our particular $\ell_1$-case will be dubbed  weak convexity: $K\subseteq {\mathbb R}^E$ is called \emph{weakly convex} if  for any $r',r''\in \vert {\covectors}\vert$ there exists
$r\in \vert {\covectors}\vert$ different from $r'$ and $r''$ such that $d(r',r'')=d(r',r)+d(r,r'')$. Convex sets of ${\mathbb R}^E$ are weakly convex, explaining the chosen name.

Weak convexity is not the weakest requirement to $|\covectors|$, which still characterizes ampleness of $\covectors$. We call a subset $K$ of ${\mathbb R}^E$ \emph{sign-convex} if for any  $r',r''\in K$ and every $e\in E$ with $r'(e)r''(e)<0$ there exists some  $r\in K$ with $r(e)=0$ and $\sign(r(f))\in \{ \sign(r'(f)),\sign(r''(f))\}$ for all $f\in E$ with $r'(f)r''(f)\ge 0$.  

\medskip\noindent
{\it {\rm(12)} Weak  and sign convexity:}  $\vert {\covectors}\vert$ is
weakly convex (equivalently, $\vert {\covectors}\vert$ is sign-convex).

\medskip
While weakly convex subsets of  ${\mathbb R}^E$ are obviously connected, sign-convexity is more admissive  because finite subsets of  ${\mathbb R}^E$  may be sign-convex. In fact we show that sign-convexity characterizes the set $\Baryc(\covectors)$ of  the barycenter maps of faces of $|\covectors|$ of ample sets $\covectors$. 

Lawrence \cite{La} showed how to derive ample sets $\covectors(K)$ from convex subsets $K$ of ${\mathbb R}^E$. There is also a
canonical construction which allows to derive ample sets $\covectors(\hcovectors)$ from certain subsets $\hcovectors$ of $\{ \pm 1, 0\}^E$. Consider the standard ordering $\prec$ of signs $-1,+1,0$ for which $-1$ and $+1$
are incomparable, $0\prec -1$ and $0\prec +1$ and endow $\{ \pm 1,0\}^E$ with the product ordering, denoted also by $\prec$. The {\it upward  closure} $\uparrow\hspace*{-0.1cm}\hcovectors$ of a subset $\hcovectors\subseteq \{ \pm 1,0\}^E$ consists of all $t'\in \{ \pm 1,0\}^E$ such that $t\prec t'$ for some $t\in \hcovectors$.  
If $\covectors$ is ample, then $\covectors$ can be retrieved from the set $\Baryc(\covectors)$ as $\covectors=\Baryc(\covectors)\cap \{ \pm 1\}^E=\uparrow\hspace*{-0.1cm}\Baryc(\covectors)\cap \{ \pm 1\}^E$.  Therefore, one may ask for which subsets $\hcovectors$ of $\{ \pm 1, 0\}^E$, the set $\uparrow\hspace*{-0.1cm}\hcovectors\cap \{ \pm 1\}^E$ is ample. We prove that sign-convexity of  the upper closure of $\hcovectors$ is sufficient to characterize  ampleness of $\uparrow\hspace*{-0.1cm}\hcovectors\cap \{ \pm 1\}^E$:

\medskip\noindent
{\it {\rm(13)} Sign-convexity of upward closure:}   $\covectors(\hcovectors)$ is ample if and only if $\uparrow\hspace*{-0.1cm}\hcovectors$ is sign-convex.

\medskip
Furthermore,  $\hcovectors$  can be retrieved from $\covectors=\uparrow\hspace*{-0.1cm}\hcovectors\cap \{ \pm 1\}^E$ as
$\Baryc(\covectors)$. The previous result can be viewed as an analog of the characterization of oriented matroids via covectors, see \cite{BjLVStWhZi}.

To characterize the subsets $\hcovectors$ of $\{ \pm 1, 0\}^E$ for which $\uparrow\hspace*{-0.1cm}\hcovectors\cap \{ \pm 1\}^E$  is ample, we need yet another relaxation of convexity in ${\mathbb R}^E$: $K\subseteq {\mathbb R}^E$ is called \emph{$0$-convex} if for any $r',r''\in K$ and every $e\in E$ with $r'(e)r''(e)<0$ there exists some  $r\in K$ with $r(e)=0$ and $\sign(r(f))\in \{ 0, \sign(r'(f)),\sign(r''(f))\}$ for all $f\in E\setminus \{ e\}$.
Each weakly convex or sign-convex subset of  ${\mathbb R}^E$ is 0-convex but the converse is not true. In case of subsets of $\{ \pm 1, 0\}^E$, 0-convexity is analogous to the ``weak elimination'' axiom for signed circuits in oriented matroids  \cite{BjLVStWhZi} and is equivalent to the following  {\it signed-circuit axiom (SCA)}:

\medskip\noindent
{\it Signed-circuit axiom:} for all  $t',t''\in {\hcovectors}$ and   $e\in E$   with
$r'(e)\cdot t''(e)=-1$ there exists  some
$t\in {\hcovectors}$ such that  $t(e)=0$ and $t(f)\in \{ 0,t'(f),t''(f)\}$
for all  $f\in E.$


\medskip\noindent
{\it {\rm(14)} 0-Convexity and (SCA):}   $\uparrow\hspace*{-0.1cm}\hcovectors\cap \{ \pm 1\}^E$ is ample if and only if $\hcovectors$ is 0-convex, i.e., $\hcovectors$ satisfies (SCA).

\medskip
Every subset $\hcovectors$  of $\{ \pm 1,0\}^E$ can be extended to the cubihedron $[{\hcovectors}]$
consisting of all cubes whose
barycenters belong to $\hcovectors$. From the previous results it follows that the cubihedron $[{\hcovectors}]$ is
path-$\ell_1$-isometric if and only if $\hcovectors$ is 0-convex.

Ample sets $\covectors$ can also be characterized via their ``cocircuits'', i.e. the set $\Cocirc(\covectors)$ of
barycentric maps of the facets (maximal faces) of the associated cubihedron
$|{\covectors}|$. From our previous characterizations  we deduce that the cocircuits of ample sets are exactly the sign maps  that satisfy the signed-circuit axiom (SCA)  and consist of pairwise minimal (for $\prec$)  maps :

\medskip\noindent
{\it {\rm(15)} Cocircuits:}   $\Cocirc(\covectors)$ satisfies (SCA) and $t,t'\in \Cocirc(\covectors)$ with $t\prec t'$ implies $t=t'$.

\medskip
Furthermore, if $\hcovectors\subset \{ \pm 1,0\}^E$ satisfies (SCA) and pairwise minimality, then $\Cocirc(\covectors(\hcovectors))=\hcovectors$.

The
characterization of cocircuits ${\RR}\subseteq  \{ \pm 1,0\}^E$ of
oriented matroids also involves maximality, the signed-circuit
axiom, but, additionally, requires  the symmetry ($r\in {\RR}$ implies that
$-r\in {\RR}$) \cite{BjLVStWhZi}. Nevertheless, the lack of symmetry implies that ample sets can differ
substantially from oriented matroids regarding their combinatorial and geometric structure.

\medskip
The primary motivation of
Lawrence in  \cite{La} was to investigate and generalize
those subsets
$${\covectors}(K) := \{ s \in \{ \pm 1\}^E:
\{ t \in K: t(e)s(e) \geq 0 \mbox{ for all } e \in E\} \neq
\varnothing \}$$ of $\{ \pm 1\}^E$ which represent the closed orthants of ${\mathbb R}^E$
intersecting a convex subset $K$ of ${\mathbb R}^E$. He (as well as Wiedemann \cite{Wi}) showed that such sets
 are ample.  However,  not every ample set encodes
the orthant intersection pattern for a convex set in Euclidean
space; see \cite{La}. It comes close, though. As we  will show below, in order
to have a full geometric representation, one has to resort to a
weaker concept of weak convexity:

\medskip\noindent
{\it {\rm(16)} Realizability by weakly convex sets:} $\covectors$  encodes the orthant
intersection pattern for some weakly convex (path-$\ell_1$-isometric) set $K$ of
$({\mathbb R}^E,\vert\vert\cdot\vert\vert_1),$ that is, a sign
vector $s$ belongs to $\covectors$ exactly when the orthant determined by $s$
also includes a point from $K.$

\medskip
We also show that for a subset $\covectors$ of $\{ \pm 1\}^E$ the  (orthogonal) projection and geometric
realization commute exactly when $\covectors$ is ample:

\medskip\noindent
{\it {\rm(17)} Commutativity of projection and geometric realization:} $|{\covectors}|_A=|{\covectors}_A|$ holds for all $A\subseteq E$.

\medskip
Furthermore, instead of $|{\covectors}|_A=|{\covectors}_A|$, it is sufficient to consider only the equality between the topological dimensions of  $|{\covectors}|_A$ and $|{\covectors}_A|$:

\medskip\noindent
{\it {\rm(18)} Equality of  dimensions:}  $\dim \,(|{\covectors}|_A)=\dim \,(|{\covectors}_A|)$ holds for
  all $A\subseteq E$.


\subsection*{(Personal) historical note} Our journey in  the world of ample/lopsided sets started in December 1995, when A. Dress sent us a fax with the inequality $\#\underline{{\X}}({\covectors})\le \#{\covectors}\le \#\overline{{\X}}({\covectors})$ and the suggestion  to investigate the set families $\covectors$ for which $\#{\covectors}=\#\overline{{\X}}({\covectors})$. Soon after we come up with the largest part of the characterizations and examples presented in \cite{BaChDrKo} and several results of the present paper. In 1996 we discovered the paper by Lawrence \cite{La} and showed that his lopsided sets and our ample sets are the same. This somehow diminished our enthusiasm, nevertheless we continued working on the subject (as certified, among other documents, by a 100 pages fax sent to each of us by A. Dress from New York in 1998). In 2006 we completed the first part of the paper, which appeared in \cite{BaChDrKo} (some preliminary results have been announced in \cite{Dr}). Soon after the publication of \cite{BaChDrKo}, each of us got a letter from R. Wiedemann with a copy of his PhD \cite{Wi} from 1986, informing us that in \cite{Wi} he rediscovered ample sets
under the name ``simple sets''. In 2012 we finished the first version of the current paper and we came across the master
thesis of S. Moran \cite{Mo}, from which we learned that the ample/lopsided sets have been also independently rediscovered by Bollob\`as and Ratcliffe \cite{BoRa} under the name ``extremal sets'' and that the inequality $\#{\covectors}\le \#\overline{{\X}}({\covectors})$ was previously established by Pajor \cite{Pa}. The inequality
$\#{\covectors}\le \#\overline{{\X}}({\covectors})$ becomes of some importance in combinatorics and computational learning theory, in particular since it implies the classical Sauer-Shelah-Perles
inequality $\#L\le \binom{n}{\le d}$ (where $d$ is the VC-dimension of $L$), which is thoroughly used in the  theory of VC-dimension. Indeed,  the simplices of $\overline{{\X}}({\covectors})$  are  the shuttered by $\covectors$ subsets of $E$  and the size of a largest simplex of $\overline{{\X}}({\covectors})$ is the VC-dimension $d$.  Do to this (and in analogy to Sauer-Shelah-Perles
inequality), we suggest to call $\#\underline{{\X}}({\covectors})\le\#{\covectors}\le \#\overline{{\X}}({\covectors})$ the \emph{Dress-Pajor inequality}.

The set families $L$ for which $\#L=\binom{n}{\le d}$ are called ``maximum sets'' \cite{GaWe} and they are ample. Maximum sets have been investigated in  \cite{GaWe} and they represent important concept/hypothesis classes for which  the \emph{sample compression conjecture} \cite{FlWa} from computational learning theory was established. This also motivated the investigation and the solution of the sample compression conjecture for ample sets  \cite{MoWa} (see also \cite{ChChMoWa} for the unlabeled version of this conjecture).

Even though this article was not submitted for publication earlier, the obtained results inspired two of us and K. Knauer to introduce the  complexes of oriented matroids (called also conditional oriented matroids and abbreviated COMs) \cite{BaChKn}. COMs are defined in terms of covectors ($\{ \pm 1, 0\}^E$-vectors) using similar axioms as the Oriented Matroids (OMs) \cite{BjLVStWhZi}, only global symmetry and the existence of the
zero sign vector, required for Oriented Matroids, are replaced by local relative conditions.  COMs represent a far reaching common generalization of ample sets, OMs, rankings, and CAT(0) Coxeter zonotopal complexes.  COMs can be endowed with the structure of a cell complex, in which each cell is an OM, and this cell complex is contractible.
The ample sets are exactly the COMs in which all cells are cubes.
The graphs of topes of COMs have been characterized by Knauer and Marc \cite{KnMa}. Currently, both ample sets and COMs are subjects of active studies.


\subsection*{Organization}
The paper is organized in the following way. In Section \ref{preliminaries} we present the main definitions used in all other sections of the paper.  Sections \ref{weak-sign-convexity} and \ref{path_intrinsic} contain preliminary results. In Section \ref{weak-sign-convexity} we investigate  metric and combinatorial relaxations of usual convexity in linear and metric spaces, which are used in all our characterizations. In Section \ref{path_intrinsic}  we discuss the intrinsic path metrics associated with an arbitrary metric space $(\MM, d)$
(this section may be skipped at first reading). The next four sections present the main results of the paper. Sections \ref{metric-characterization}  and \ref{projection}
present characterizations of ample cubihedra via metric conditions, weak and sign-convexities,  and via projections. Then the geometric
realization of ample sets in terms of intersection pattern with orthants is established in Section \ref{realization}. Section \ref{circuit-cocircuit} investigates the concepts of circuits and cocircuits and their relation to the geometric structure of ample cubihedra.

\section{Preliminaries} \label{preliminaries}

\subsection{Maps and sign maps} Given a finite set $E$, we denote by $\#E$ its cardinality,
by $\{ \pm 1\}^E$  the set  of all maps from  $E$
into  $\{ \pm 1\}=\{ -1,+1\}$, by $\{ \pm 1,0\}^E$  the set  of all maps from  $E$
into  $\{ \pm 1,0\}$, and by ${\mathbb R}^E$ the vector space consisting of all maps $r$
from $E$ into $\mathbb R$. Clearly, $\{ \pm 1\}^E\subset \{ \pm 1,0\}^E\subset {\mathbb R}^E$.
For $r \in {\mathbb R}^E$ and $e\in E$, we denote by $\sign(r(e))\in \{ -1,0,+1\}$ the sign of $r(e)$.
To each map $r\in {\mathbb R}^E$ we associate its \emph{support} $\supp(r)=\{ e\in E: r(e)\ne 0\}$ and its  \emph{sign vector} $\sign(r)\in \{ \pm 1, 0\}^E$ whose coordinates are $\sign(r(e)), e\in E$.
For any $r',r''\in {\mathbb R}^E$, we denote by $\Delta(r',r'')=\{ e\in E: r'(e)\ne r''(e)\}$ their \emph{difference set} and by
$d(r',r'')$ their $\ell_1$-distance $||r',r''||_1=\sum_{e\in E} |r'(e)-r''(e)|=\sum_{e\in \Delta(r',r'')} |r'(e)-r''(e)|$. Let also $[r',r'']=\{ r\in {\mathbb R}^E: d(r',r)+d(r,r'')=d(r',r'')$. Note that $[r',r'']$ is an axis-parallel box.

Each  $e\in E$ defines a coordinate hyperplane $H_e=\{ r\in {\mathbb R}^E: r(e)=0\}$ of ${\mathbb R}^E$, which partitions ${\mathbb R}^E-H_e$ into \emph{positive} and \emph{negative} open halfspaces.   The set ${\mathcal H}=\{ H_e: e\in E\}$ is a \emph{central arrangement} of hyperplanes; ${\mathcal H}$  indices a partition ${\mathcal R}$ of ${\mathbb R}^E$ into open regions and recursively into regions defined by  the intersections of some  hyperplanes of ${\mathcal H}$. More specifically, each region $R\in \mathcal R$ is the intersection of a set $H_e, e\in E'$ of  hyperplanes of ${\mathcal H}$ and a set of positive or negative halfspaces, one halfspace for each  $H_f, f\in E-E'$.  The dimension of each $R$ varies between  $\#E$ and $0$. The regions of $\mathcal R$ of maximal dimension are the $2^{\# E}$ open orthants of ${\mathbb R}^E$.

For every point $r$ in ${\mathbb R}^E$, the sign vector $\sign(r)$  is the sign map
$s\in \{ \pm 1,0\}^E$ such that $s(e)=0$ if $r\in H_e$, $s(e)=+1$ if $r$ belongs to the positive halfspace of $H_e$ and $s(e)=-1$ if $r$ belongs to the negative halfspace of $H_e$. All points $r$ belonging to the same region $R$ of ${\mathcal R}$ have the same sign vector and any two points belonging to distinct regions of ${\mathcal R}$ have distinct sign vectors. In particular, all points belonging to the same open orthant $\OO$ of ${\mathbb R}^E$ have the same sign vector
$s\in \{ \pm 1\}^E$ and $\OO$ can be retrieved from $s$ as 
\begin{align*}
\OO=\{ r\in {\mathbb R}^E: ~r(e)\cdot s(e)>0  \mbox{ for all } e\in E\}.
\end{align*}
The closure $\widetilde{\OO}$ of $\OO$ is a closed orthant of ${\mathbb R}^E$ and $\widetilde{\OO}$ can be also identified by the sign map $s$ as $\widetilde{\OO}=\OO(s)$, where
\begin{align*}
\OO(s):=\{ r\in {\mathbb R}^E: ~r(e)\cdot s(e)\ge 0  \mbox{ for all } e\in E\}.
\end{align*}
Conversely, given $r\in {\mathbb R}^E,$ the set
\begin{align*}
\Sign(r):=\{ s\in \{ \pm 1\}^E: ~r(e)\cdot s(e)\ge 0 \mbox{ for all } e\in E\},
\end{align*}
indicates the closed orthants of ${\mathbb R}^E$ to which $r$ belongs.

\subsection{The hypercube $\HH(E)$ and  its faces} By $\HH(E)$ we denote the \emph{solid hypercube} $[-1,+1]^E\subset {\mathbb R}^E$ and consider $\HH(E)$ as a cube complex. By $H^{(1)}(E)$ we denote the graph representing the \emph{1-skeleton} of $\HH(E)$, i.e., the graph with vertex-set
$\{ \pm 1\}^E$  and whose edge set consists of all pairs $\{ s,s'\}$ of $\{ \pm 1\}^E$ such that $d(s,s')=2$ (equivalently, $\Delta(s,s')=1$). This implies that  the standard graph distance  in $H^{(1)}(E)$ between any two vertices $s',s''\in \HH^{(0)}$ is equal to $\frac{1}{2}d(s',s'')$.

\begin{definition} [Faces of $\HH(E)$ and their barycenters] A {\it face} $\FF$ of the hypercube $\HH(E)$ is a $A$-fiber of the form
$${\FF}=\HH(A)\times s_0 \mbox{ for some } s_0\in \{ \pm 1\}^{E-A}$$
with $A\subseteq E.$ By convention,  $\HH(E)$ is its $E$-fiber and its vertices are the faces that are $\varnothing$-fibers. The \emph{barycenter} of the face $\FF$ is the map $s\in \{ \pm 1,0\}^E$ such that $s(e)=0$ if $e\in A$ and $s(e)=s_0(e)\in \{ -1,+1\}$  if $e\in E-A$.     We say that two
faces $F',F''$ of $\HH(E)$ are {\it parallel} if there exist $A\subseteq E$ and  $s'_0,s''_0\in \{ \pm 1\}^{E-A}$ such that ${\FF}'=\HH(A)\times s'_0$ and ${\FF}''=\HH(A)\times s''_0$.   Parallel faces thus arise by intersecting
the hypercube with parallel (affine) $A$-{\it planes}
$$\{ p\in {\mathbb R}^E: ~p|_{E-A}=c\} \mbox{ for } c\in \HH(E-A).$$
For each $r\in \HH(E)$ denote by $F(r)$ the (necessarily unique) smallest face of $\HH(E)$ containing $r$.
\end{definition}

For each point $r$ of $\HH(E)$, the face $F(r)$  is determined by the coordinates $r(e)$ ($e\in E$) for which $-1<r(e)<+1:$
$$F(r)=\HH(E(r))\times r|_{E-E(r)}, \mbox{ where }$$
$$E(r):=\{ e\in E: ~-1<r(e)<+1\}, \mbox{ and }$$
$$r|_{E-E(r)}\in \{ \pm 1\}^{E-E(r)}.$$
The barycenter of the face $F(r)$ is given by the map from $E$ to $\{ \pm 1,0\}$ that is the zero map on $E(r)$ and coincides with $r|_{E-E(r)}$ elsewhere. 

\begin{figure}[t]
\vspace*{-3cm}
\centering\includegraphics[scale=0.45]{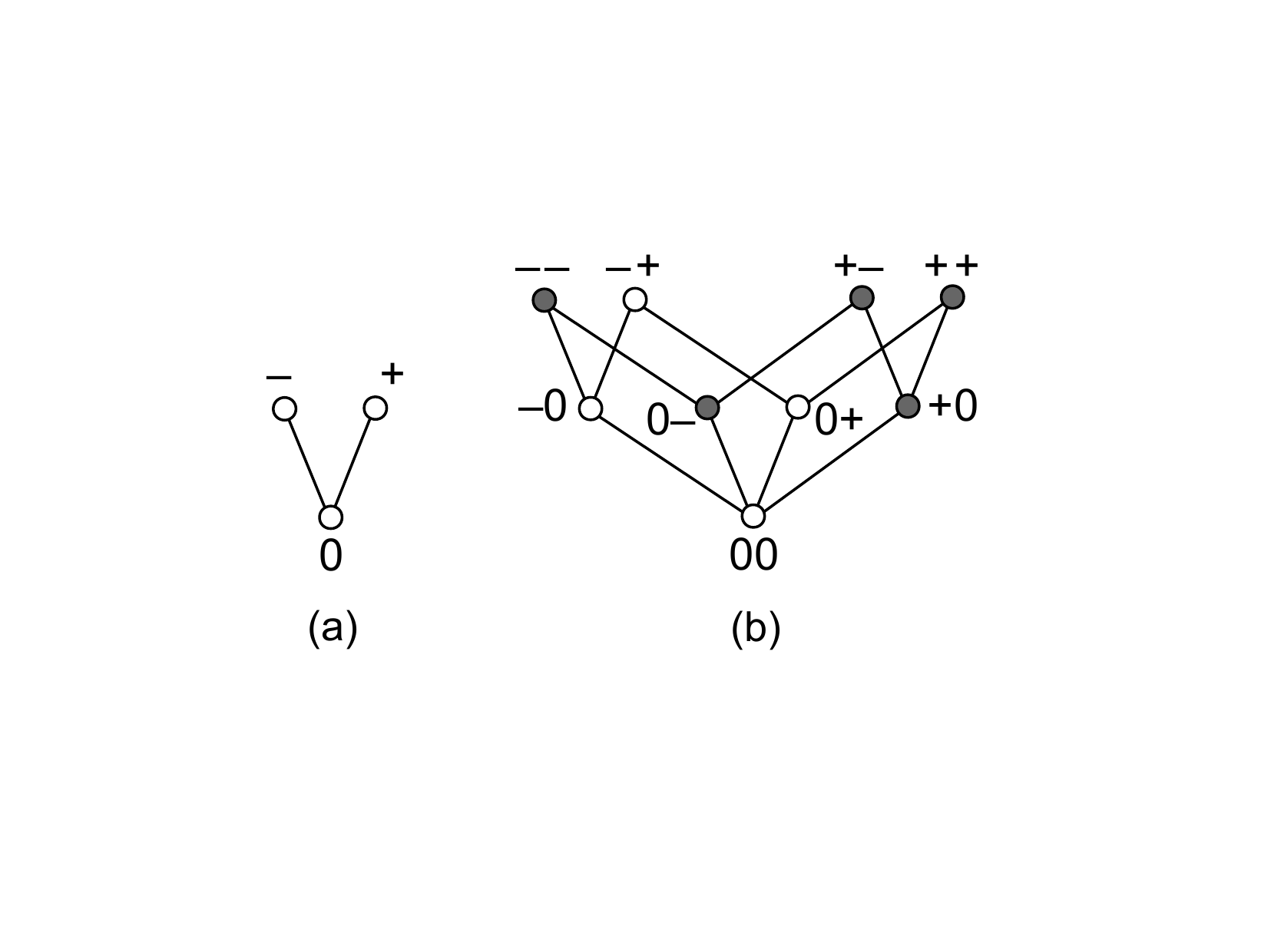}
\vspace*{-3cm}
\caption{(a) Ordering of signs. ~(b)  Product ordering on $\{ \pm 1,0\}^2$; the shaded nodes correspond to the barycenter maps of an ample set. }
\label{figure1}
\end{figure}

The set of all barycenters of  faces of $\HH(E)$ is  $\{ \pm 1,0\}^E$.  In order to express inclusion of faces of $\HH(E)$ in terms of the corresponding barycenter maps we use the standard ordering $\prec$ of signs $-1,+1,0$ for which $-1$ and $+1$
are incomparable, $0\prec -1$ and $0\prec +1$; see Fig. ~1(a).  The product ordering $\prec^E$ on $\{ \pm 1,0\}^E$ will also be denoted by $\prec$. The undirected Hasse diagram of $(\{ \pm 1,0\}^E,\prec)$ is a grid graph (viz., the Cartesian $E$-power of a path
with two edges) and will be denoted by $G(\{ \pm 1,0\}^E)$; see Fig. ~1(b) for $\# E=2$. Thus, $t_1\prec t_2$ for two maps $t_1,t_2\in \{ \pm 1,0\}^E$ holds if and only if $t_1(x)\in \{ 0,t_2(x)\}$ for all $x\in E,$ or equivalently, if for the associated faces the inclusion $F(t_2)\subseteq F(t_1)$ holds.

\subsection{Subsets of sign maps, their cubihedra and barycentric completions}
For a subset ${\covectors}$ of $\{ \pm 1\}^E$, we denote by ${\covectors}^*:=\{ \pm 1\}^E-{\covectors}$ the set-theoretic complement of $\covectors$ and by  $G({\covectors})$ the subgraph of  ${\HH}^{(1)}(E)$ induced by $\covectors$,
i.e., the graph with vertex-set $L$ and edge-set consisting of all edges of ${\HH}^{(1)}(E)$ between two vertices of $\covectors$ (again we suppose that the edges of $G(L)$ have length 2). The set $\covectors$ is
called {\it connected}
if $G({\covectors})$ is connected, and it is called {\it isometric} if every
pair of vertices $s',s''$ of $\covectors$ can be connected in $G({\covectors})$
by a path of length $d(s',s'').$

\begin{definition} [The cubihedron] The \emph{geometric realization} (a {\it cubihedron} for short) of  a set $L\subseteq \{ \pm 1\}^E$ is the cube complex $|{\covectors}|\subseteq \HH(E)$
consisting of all faces $\FF$ of the hypercube ${\HH}(E)$   for which ${\FF}\cap \{ \pm 1\}^E\subseteq {\covectors}.$ The vertices of $|{\covectors}|$ are  the elements of ${\covectors}$
and  the 1-skeleton $|{\covectors}|^{(1)}$ of $|{\covectors}|$ coincides with $G(\covectors)$. The \emph{facets} of $|\covectors|$ are the maximal by inclusion faces of $|\covectors|$.
\end{definition}

Notice that for each point $r$ in  $|\covectors|$, its smallest face $F(r)$ in the hypercube $\HH(E)$ necessarily belongs to $|\covectors|$.

\begin{definition} [The barycentric completion] For a subset $L$ of $\{ \pm 1\}^E$, the subset of $\{ \pm 1, 0\}^E$ consisting  of all barycenters of the faces of $|\covectors|$  is called the \emph{barycentric completion} of $\covectors$ and is denoted by $\Baryc(\covectors)$ (or $\Baryc(|\covectors|)$).
\end{definition}

\begin{definition} [The upward closure and the set of minima] \label{def:upper-closure} For a subset $\hcovectors\subseteq \{ \pm 1,0\}^E$,  the {\it upward  closure} $\uparrow\hspace*{-0.1cm}\hcovectors$ of $\hcovectors$ relative to the ordering $\prec$ is defined by
$$\uparrow\hspace*{-0.1cm}\hcovectors=\{ t'\in \{ \pm 1,0\}^E: ~t\prec t' \mbox{ for some } t\in \hcovectors\}.$$
Then $\hcovectors\subseteq \{ \pm 1,0\}^E$ is called \emph{upward closed} if $\uparrow\hspace*{-0.1cm}\hcovectors=\hcovectors$.

Denote also by $\Min(\hcovectors)=\{ t\in \hcovectors: ~t'\prec t \mbox{ implies } t'=t \mbox{ for all } t'\in \hcovectors\}$ the set of all minimal elements of the poset $(\hcovectors, \prec)$.
\end{definition}

We continue with two simple properties of upward closure:

\begin{lemma} \label{upper-set-baryc} For any subset $\hcovectors$ of $\{ \pm 1,0\}^E$, $\Baryc(\uparrow\hspace*{-0.1cm}\hcovectors\cap \{ \pm 1\}^E)=\uparrow\hspace*{-0.1cm}\hcovectors$.
\end{lemma}

\begin{proof} Indeed, $\uparrow\hspace*{-0.1cm}\hcovectors$ consists of the barycenters $t'$ of faces $F(t')$ included in faces $F(t)$ of $\HH(E)$ with $t\in \hcovectors$. These are exactly the barycenters of faces $F$ of $\HH(E)$ for which $F\cap\{ \pm 1\}^E\subseteq \uparrow\hspace*{-0.1cm}\hcovectors\cap \{ \pm 1\}^E$, i.e.,
$\Baryc(\uparrow\hspace*{-0.1cm}\hcovectors\cap \{ \pm 1\}^E)$.
\end{proof}

\begin{lemma} \label{upper-set-baryc} For any subset $\covectors$ of $\{ \pm 1\}^E$, $\Baryc(\covectors)$ is an upward closed subset of $\{ \pm 1,0\}^E$.
\end{lemma}

\begin{proof} If $t\in \Baryc(\covectors)$ and $t\prec t'$, then $t'$ is the barycenter of the face $F(t')$ contained in $F(t)$. Since $F(t)$ belongs to $|L|$, $F(t')$ also belongs to $|L|$, whence $t'\in \Baryc(L)$. This shows that
$\uparrow\hspace*{-0.1cm}\Baryc(\covectors)\subseteq \Baryc(\covectors)$. The converse inclusion is obvious.
\end{proof}



\subsection{Intersection patterns and projections of subsets  of ${\mathbb R}^E$} To any subset $K$ of ${\mathbb R}^E$ we associate two sets of sign maps $L(K)\subseteq \{ \pm 1\}^E$ and $J(K)\subseteq \{ \pm 1, 0\}^E$,  encoding the intersection patterns of $K$ with the closed orthants of ${\mathbb R}^E$ and with the regions of the partition $\mathcal R$ of ${\mathbb R}^E$ induced by the arrangement $\mathcal H$, respectively.

\begin{definition} [$L(K)$ and $J(K)$]
Given any subset $K$ of ${\mathbb R}^E,$ let
\begin{align*}
L(K):=\bigcup_{r\in K}\Sign(r)=\{ s\in \{ \pm 1\}^E: ~~\exists r\in K \mbox{ with } r(e)\cdot s(e)>0 ~\forall e\in \supp(r)\}
\end{align*}
and
\begin{align*}
J(K):=\bigcup_{r\in K}\sign(r)=\{ t\in \{ \pm 1,0\}^E: ~\exists r\in K \mbox{ with } \sign_0(r)=t\}.
\end{align*}
\end{definition}

Thus $s\in \{\pm 1\}^E$ belongs to $L(K)$ exactly when $K\cap \OO(s)\ne\varnothing$, i.e. if exists some $r\in K$ with $r(e)\cdot s(e)\ge 0$ for all $e\in E$. Analogously, $t\in \{ \pm 1,0\}^E$ belongs to $J(K)$ exactly when there exists some $r\in K$ with $r(e)\cdot t(e)>0$ for all  $e\in \supp(r)$ and $t(e)=0$ for all $e\in E-\supp(r)$. Note also that $L(L(K))=L(K)$ and $J(J(K))=J(K)$.

To any subset $\hcovectors$ of $\{ \pm 1,0\}^E$  one associates the set of sign maps
\begin{align*}
L(\hcovectors)=\bigcup_{r\in \hcovectors}\Sign(r)=\uparrow\hspace*{-0.1cm}{\hcovectors}\cap \{ \pm 1\}^E=L(\uparrow\hspace*{-0.1cm}\hcovectors)
\end{align*}
and the corresponding geometric realization $|L(\hcovectors)|$ of $L(\hcovectors)$. This realization can be compared to the cube complex $[\hcovectors]$, which is defined as the union of the smallest faces $F(t)$
of $\HH(E)$ containing $t$ for $t\in \hcovectors$:
$$[\hcovectors]:=\bigcup_{t\in \covectors} F(t)~\subseteq ~|L(\hcovectors)|.$$
One can retrieve the upward closure $\uparrow\hspace*{-0.1cm}{\hcovectors}$ from $[\hcovectors]$ as
\begin{align*}
\uparrow\hspace*{-0.1cm}{\hcovectors}=\hcovectors([\hcovectors])=\Baryc([\hcovectors]).
\end{align*}

We now introduce the notation $K^A$ and $K_A$ for  subsets $K$  of $\HH(E)$  in analogy to $\covectors^A$ and $\covectors_A$ for subsets $\covectors$ of $\{ \pm 1\}^E.$ For a subset $K$ of ${\mathbb R}^E$ and a subset $A$ of $E$, we denote by $K_A$ the image of $K$ relative to  to the orthogonal projection of $K$ onto the $(E-A)$-plane:
$$K_A:=\{ r|_{E-A}: ~r\in K\}\subseteq H(E-A).$$
For a subset $K$ of $\HH(E)$ and $A\subseteq E$, the set
$$K^A:=\{ r|_{E-A}: ~\HH(A)\times r|_{E-A}\subseteq K\}\subseteq \HH(E-A)$$
encodes the location of the $A$-fibers in $K$ that are also $A$-fibers (faces) of $\HH(E).$


\subsection{Intrinsic metric}
Let $(M,\rho)$ be a metric space, $I\subset {\mathbb R}$ a non-empty interval
and $\gamma: I\rightarrow M$ a curve (i.e. a continuous map). We define the length $\ell(\gamma)\in [0,\infty]$
of $\gamma$ by $\ell(\gamma)=\sup\sum_{i=1}^k \rho(\gamma(t_{i-1},t_i)$, where the supremum is taken over all $k\in {\mathbb N}$ and all sequences $t_0\le t_1\le \ldots t_k$ in $I$.  We say that $\gamma$ is \emph{rectifiable} if
$\ell(\gamma)<\infty$. The \emph{intrinsic  metric} associated with $\rho$ is the function $\delta_{M,\rho}: M\times M \rightarrow [0,\infty]$ defined by $\delta^{M,d}(x,y)=\inf \ell(\gamma)$, where the infimum is taken over all rectifiable curves $\gamma:[0, 1]\rightarrow M$  from $x$ to $y$,
i.e. $\gamma(0) = x, \gamma(1) = y$. $(M, \rho)$ is called a \emph{length space} if $\delta^{M,d}=\rho$ (see the Appendix for other related notions and results).

A \emph{geodesic} of $(M,\rho)$ is the image of a continuous map $\gamma: [0,\alpha]\rightarrow M$ with $\alpha=\rho(r',r'')$ such that $\gamma(0)=r', \gamma(\alpha)=r''$ and $\rho(\gamma(s),\gamma(t))=|s-t|$ for all $s,t\in [0,\alpha]$. A metric space $(M,\rho)$ is a \emph{geodesic space} \cite{BrHa} if any two points $r',r''$ can be connected in $M$ by a geodesic. In this paper, we consider the geodesic metric space $({\mathbb R}^E, d)$, where $d$ is the $\ell_1$-distance. An \emph{$\ell_1$-geodesic} is a geodesic of  $({\mathbb R}^E, d)$.

A metric space $(\MM,d)$ is {\it Menger-convex} if for any two distinct points $x,y\in \MM$ there exists some point $z$ {\it between} $x$ and $y,$ that is, $z$ belongs to the {\it segment} $$[x,y]_{\MM}=\{z\in \MM: ~d(x,z) + d(z,y) = d(x,y)\}$$
such that, in addition, $z$ is different from $x$  and $y$. Menger \cite{Me} has shown that Menger-convexity in a complete metric space $(\MM,d)$
entails that $(\MM,d)$ is geodesic (see also \cite{Aro} or \cite[\S 18.5]{Ri}).

\section{Weak and sign convexities} \label{weak-sign-convexity} In this section we consider several relaxations of classical convexity in metric  spaces.
We consider every subset $K$ of ${\mathbb R}^E$ as a metric space $(K,d)$ relative to the metric induced on $K$ by the $\ell_1$-metric $d$ (we denote the induced metric also by $d$). Recall that a subset $K$ of  ${\mathbb R}^E$  is called \emph{convex} if it contains the linear segment between any two points $r',r''\in K$. Analogously,
 $K\subseteq {\mathbb R}^E$ is called \emph{$\ell_1$-convex} if it contains all $\ell_1$-geodesics between any two points $r',r''\in K$. Since the linear segments are  $\ell_1$-geodesics, every $\ell_1$-convex set is convex.
 For any $e\in E$, the coordinate hyperplane $H_e$, the open halfspaces $H^-_e,H^+_e$, and the closed halfspaces $\overline{H}^-_e,\overline{H}^+_e$ defined by $H_e$ are $\ell_1$-convex. Furthermore, the faces of the hypercube $\HH(E)$ are also $\ell_1$-convex. It is known (and easy to prove) that the $\ell_1$-convex subsets of ${\mathbb R}^E$ are gated in the following sense. A subset $K$ of any metric space $({\mathbb R}^E,d)$ is called {\it gated} \cite{DrScha}  if for
every point $r\in {\mathbb R}^E$ there exists a (necessarily unique) point $r'\in K,$ the {\it gate} of $r$ in $K$, for which $$d(r,q)=d(r,r')+d(r',q) \mbox{ for all } q\in K.$$

 Now, we consider  weaker versions of convexity.

\begin{definition} [Weak convexity] A subset $K$ of  ${\mathbb R}^E$ is called \emph{weakly convex} if $(K,d)$ is complete and Menger-convex, i.e., if for any two distinct points $r',r''\in K$ there exists some point $r\ne r',r''$ such that $d(r',r'')=d(r',r)+d(r,r'')$.
\end{definition}

Clearly, every convex subset of ${\mathbb R}^E$ is weakly convex. The following lemma can be viewed as a reformulation of the definition of weak convexity and follows from Menger's  theorem \cite{Me}
mentioned above:

\begin{lemma} \label{weak-convexity} For a subset $K$ of ${\mathbb R}^E$ such that $(K,d)$ is complete, the following conditions are equivalent:
\begin{itemize}
\item[\rm{(i)}] $K$ is weakly convex;
\item[\rm{(ii)}] $(K,d)$ is a length space;
\item[\rm{(iii)}] $K$ is {\it path-$\ell_1$-isometric} in the sense that the restriction of the $\ell_1$-metric $d$ on $K$ constitutes the intrinsic metric $\delta^{K,d}$ of $K$.
\end{itemize}
\end{lemma}

We continue with two combinatorial relaxations of weak convexity:

\begin{definition} [Sign convexity]
A subset $K$ of ${\mathbb R}^E$ is called \emph{sign-convex} if for any two distinct points $r',r''\in K$ and every $e\in E$ with $r'(e)r''(e)<0$ there exists some point $r\in K$ with $r(e)=0$ and $\sign(r(f))\in \{ \sign(r'(f)),\sign(r''(f))\}$ for all $f\in E$ with $r'(f)r''(f)\ge 0$ (and no condition on $\sign(r(f))$ in case $r'(f)r''(f)<0$ and $f\ne e$).
\end{definition}

\begin{lemma} \label{weakconv->signconv} Every weakly convex (and hence every convex) subset $K$ of  ${\mathbb R}^E$ is sign-convex.
\end{lemma}

\begin{proof} Pick any $r',r''\in K$ and $e\in E$ with $r'(e)r''(e)<0$. Since $K$ is weakly convex, $r'$ and $r''$ can be connected in $K$ by an $\ell_1$-geodesic $\gamma(r',r'')$. Consider the coordinate hyperplane $H_{e}$  defined by $e$.  Since  $r'(e)r''(e)<0$, $r'$ and $r''$ belong to distinct open halfspaces defined by $H_e$. Therefore $\gamma(r',r'')$ intersects $H_e$ in a point $r$. Then $r(e)=0$, thus $\sign(r(e))=0$. Now, consider any $f\in E$ such that $r'(f)r''(f)\ge 0$. This implies that  $r'$ and $r''$ belong to the same closed halfspace defined by the hyperplane $H_f$, say $r',r''$ belong to the closed negative halfspace $\overline{H}_f^-$. Since  $\overline{H}_f^-$ is $\ell_1$-convex, the geodesic $\gamma(r',r'')$ as well as the point $r$ also belong to $\overline{H}_f^-$. Therefore $r(f)\le 0$ and thus $\sign(r(f))\in \{ -1,0\}$. If one of the points $r',r''$ belong to $H_f$ and the second point belong to the open halfspace $H^-_f$, then $\{ \sign(r'(f)),\sign(r''(f))\}=\{ -1,0\}$, yielding $\sign(r(f))\in \{ \sign(r'(f)),\sign(r''(f))\}$. If both points $r',r''$ belong
to $H_f$, then $r$ also belongs to $H_f$ because $H_f$ is $\ell_1$-convex and we get $\sign(r(f))=\sign(r'(f))=\sign(r''(f))=0$. Finally, if both points $r',r''$ belong
to the open halfspace $H^-_f$, then $r$ also belongs to $H^-_f$ because $H^-_f$ is $\ell_1$-convex and we get $\sign(r(f))=\sign(r'(f))=\sign(r''(f))=-1$. In all cases, $r'(f)r''(f)\ge 0$ implies $\sign(r(f))\in  \{ \sign(r'(f)),\sign(r''(f))\}$ as required.
\end{proof}

Weakly convex subsets of ${\mathbb R}^E$ are path $\ell_1$-isometric, thus they are connected. On the other hand, sign-convex sets are not necessarily connected, in particular, they may be finite. For example, the set $\{ \pm 1,0\}^E$ is sign-convex. The following result relates the sign-convexity of subsets of ${\mathbb R}^E$ with the sign-convexity of subsets of $\{ \pm 1,0\}^E$.

\begin{lemma}\label{sign-convex} A subset $K$ of ${\mathbb R}^E$ is sign-convex if and only if the subset $J(K)$ of $\{ \pm 1,0\}^E$ is sign-convex. Consequently, $K\subseteq {\mathbb R}^E$ is sign-convex if and only if there exists $S\subseteq \{ \pm 1,0\}^E$ such that $\hcovectors(K)=\hcovectors(S)$.
\end{lemma}

\begin{proof} Pick any $r',r''\in K$ and let $t'=\sign(r'), t''=\sign(r'')$. By definition, for any $f\in E$, the inequality $t'(f)t''(f)\ge 0$ holds if and only if the inequality $r'(f)r''(f)\ge 0$ holds. First, suppose that $K$ is sign-convex. To show that $J(K)$ is sign-convex,  pick any $t',t''\in J(K)$ and $e\in E$ such that  $t'(e)=-1$ and $t''(e)=+1$.  Let $r',r''\in K$ be such that $\sign(r')=t'$ and $\sign(r'')=t''$. Then $r'(e)<0$ and $r''(e)>0$. Since $K$ is sign-convex, there must exist some $r\in K$ such that
$r(e)=0$ and $\sign(r(f))\in \{ \sign(r'(f)),\sign(r''(f))\}$ for any $f\in E$ such that $r'(f)r''(f)\ge 0$.
Let $t=\sign(r)$. Then $t(e)=\sign(r(e))=0$.  Furthermore, since $t'(f)t''(f)\ge 0$ if and only if $r'(f)r''(f)\ge 0$, we get $t(f)=\sign(r(f))\in \{ \sign(r'(f)),\sign(r''(f))\}=\{ t'(f),t''(f)\}$
for any such $f\in E$ such that $t'(f)t''(f)\ge 0$.

Conversely, suppose that $K$ is a subset of ${\mathbb R}^E$ such that $J(K)$ is sign-convex. To prove that $K$ is sign-convex let $r',r''\in K$ and $e\in E$ such that $r'(e)r''(e)<0$. Let $t'=\sign(r')$ and $t''=\sign(r'')$. Then $t'(e)t''(e)<0$. Since $t',t''\in J(K)$ and $J(K)$ is sign-convex,
there exists $t\in J(K)$ such that$t(e)=0$ and $t(f)\in \{ t'(f),t''(f)\}$ for all $f\in E$ such that $t'(f)t''(f)\ge 0$. Pick any $r\in K$ such that $\sign(r)=t$. Then $\sign(r(e))=t(e)=0$, whence $r(e)=0$. Furthermore,  since $t'(f)t''(f)\ge 0$ if and only if $r'(f)r''(f)\ge 0$, we get  $\sign(r(f))=t(f)\in \{ t'(f),t''(f)\}=\{\sign(r'(f)),\sign(r''(f))\}$ for any $f\in E$ such that $r'(f)r''(f)\ge 0$. This concludes the proof.
\end{proof}

\begin{definition} [$0$-convexity]
A subset $K$ of ${\mathbb R}^E$ is called \emph{$0$-convex} if for any two points $r',r''\in K$ and every $e\in E$ with $r'(e)r''(e)<0$ there exists some point $r\in K$ with $r(e)=0$ and $\sign(r(f))\in \{ 0, \sign(r'(f)),\sign(r''(f))\}$ for all $f\in E\setminus \{ e\}$.
\end{definition}

Notice that 0-convexity of subsets of  $\{ \pm 1, 0\}^E$  comes from the following axiom, which is analogous to the ``weak elimination'' axiom for signed circuits in oriented matroids (see \cite[Definition 3.2.1]{BjLVStWhZi}). We say that a subset ${\hcovectors}\subseteq \{ \pm 1, 0\}^E$ satisfies the {\it signed-circuit axiom (SCA)} if the following condition holds:
\begin{align*}
\mbox{ for all } t',t''\in {\hcovectors}, \mbox{ and }  e\in E  \mbox{ with  }
t'(e)\cdot t''(e)=-1 \mbox{ there exists  some}\tag{\mbox{SCA}}\\
t\in {\hcovectors} \mbox{ such that } t(e)=0 \mbox{ and }
t(f)\in \{ 0,t'(f),t''(f)\}  \mbox{ for all } f\in E.
\end{align*}
Consequently,  a subset $\hcovectors$ of $\{\pm 1, 0\}^E$ is  $0$-convex if and only if it satisfies (SCA).

The difference between sign-convexity and 0-convexity is that for $f\in E$ with $r'(f)r''(f)=+1$,  sign-convexity requires that $r(f)=r'(f)=r''(f)$, while 0-convexity allows $r(f)$ to be $0$ or $r'(f)=r''(f)$. If $r'(f)r''(f)=0$ and say $r'(f)=0$, both conditions require that $r(f)\in \{ 0,r''(f)\}$. Finally, if  $r'(f)r''(f)=+1$, then both sign- and 0-convexities  do not impose any constraint on the sign of $r(f)$. Therefore, the following holds (where the second assertion follows from  the first assertion and Lemma \ref{sign-convex}):

\begin{lemma}\label{sign-convexity->SCA} If $K$ is a sign-convex  subset $K$ of  ${\mathbb R}^E$, then the sets $K$ and   $\hcovectors(K)$ are 0-convex.
\end{lemma}

Any subset $\hcovectors$ of $\{ \pm 1,0\}^E$ containing the zero map $(0,\ldots,0)$ is $0$-convex but not necessarily sign-convex, showing that $0$-convexity is weaker than sign-convexity. However, for upward closed subsets of  $\{ \pm 1,0\}^E$, sign-convexity and 0-convexity are equivalent.  Moreover, each of them is equivalent to a ``Menger-convexity type'' condition, requiring that the sign map $r$ in the definitions of sign- and 0-convexity belongs to the box $[r',r'']$.


\begin{proposition} \label{sign-convex-upper-closed} For an upward closed subset $\hcovectors$
of $\{ \pm 1,0\}^E$ the following conditions are equivalent:
\begin{itemize}
\item[\rm(i)] $\hcovectors$ is sign-convex;
\item[\rm(i$'$)] for any $t_1,t_2\in \hcovectors$ and $e\in E$ with $t_1(e)t_2(e)=-1$ there exists some $t_0\in [t_1,t_2]\cap \hcovectors$ with $t_0(e)=0$ and $t_0(f)\in \{ t_1(f),t_2(f)\}$ for any $f\in E$ such that $t_1(f)t_2(f)\ge 0$;
\item[\rm(ii)] $\hcovectors$ is $0$-convex;
\item[\rm(ii$'$)] for any $t_1,t_2\in \hcovectors$ and $e\in E$ with $t_1(e)t_2(e)=-1$ there exists some $t_0\in [t_1,t_2]\cap \hcovectors$ with $t_0(e)=0$ and $t_0(f)\in \{ 0, t_1(f),t_2(f)\}$ for any $f\in E$ such that $t_1(f)t_2(f)\ge 0$.
\end{itemize}
Every upward closed sign-convex (0-convex) subset $\hcovectors$ of $\{ \pm 1,0\}^E$ induces an isometric
subgraph $G(\hcovectors)$ of the grid $G(\{ \pm 1, 0\}^E)$.
\end{proposition}

\begin{proof} The implication  {\rm(i$'$)}$\Rightarrow${\rm(i)} is  trivial and the implication {\rm(i)}$\Rightarrow${\rm(ii)}  follows from Lemma \ref{sign-convexity->SCA}. To prove  {\rm(ii)}$\Rightarrow${\rm(ii$'$)}, suppose that $\hcovectors$ is an upward closed 0-convex subset of $\{ \pm 1, 0\}^E$. Pick any $t_1,t_2\in \hcovectors$ and $e\in E$ with $t_1(e)t_2(e)=-1$.
$0$-Convexity yields some $t\in \hcovectors$ with $t(f)\in \{ 0,t'(f),t''(f)\}$ for all $f\in E-\{ e\}$ and $t(e)=0.$ Then the map $t_0$ defined by
$$
t_0(f):=\begin{cases}
t'(f) &\mbox{ if } t(f)=0 \mbox{ and } f\ne e,\\
t(f) &\mbox{ otherwise }
\end{cases}
$$
dominates $t$ and hence belongs to $\uparrow\hspace*{-0.1cm}\hcovectors\cap [t',t'']=\hcovectors\cap [t',t''].$

To prove  {\rm(ii$'$)}$\Rightarrow${\rm(i$'$)}, let $t_1,t_2\in \hcovectors$ and $e\in E$ with $t_1(e)t_2(e)=-1$.  Condition {\rm(i$'$)} yields some $t\in [t_1,t_2]\cap \hcovectors$ with $t(e)=0$ and $t(f)\in \{ 0, t_1(f),t_2(f)\}$ for any $f\in E\setminus \{ e\}$. Now, pick any $f\in E\setminus \{ e\}$ such that $t_1(f)t_2(f)\ge 0$. If $t(f)\notin \{ t_1(f),t_2(f)\}$, this implies that $t(f)=0$ and $t_1(f),t_2(f)\in \{ \pm 1\}^E$. Since $t_1(f)t_2(f)\ge 0$, we conclude that $t_1(f)=t_2(f)$. Since $t\in [t_1,t_2]$ and $t(f)=0$, we obtain a contradiction.

Finally, we prove that for every upward closed sign-convex subset $\hcovectors$ of $\{ \pm 1,0\}^E$,  $G(\hcovectors)$ is an isometric subgraph of  $G(\{ \pm 1, 0\}^E)$. Suppose by way of contradiction that there were some distinct $t_1,t_2\in\hcovectors$ which are not adjacent in the
graph $G(\{ \pm 1,0\}^E)$ such that the box $[t_1,t_2]$ in $\HH(E)$ intersects $\hcovectors$ only in $t_1$ and $t_2.$ If for some $e\in E$ we have $t_1(e)\cdot t_2(e)=-1,$ then there exists some $t_0\in [t_1,t_2]\cap\hcovectors$ with $t_0(e)=0$ by condition {\rm(i$'$)}. Then $t_0\notin \{ t_1,t_2\}$ and we obtain a contradiction with the initial hypothesis. Therefore $t_1(f)\cdot t_2(f)\ge 0$ for all $f\in E,$ that is, all coordinates of $t_1$ and $t_2$ have comparable signs. Consequently,
the join $t$ of $t_1$ and $t_2$ exists in the ordered set $(\{ \pm 1,0\}^E,\prec)$ and is given by $t(f)=\mbox{min}\{ t_1(f),t_2(f)\}$ for all $f\in E.$ Then
$t\in [t_1,t_2]\cap \uparrow\hspace*{-0.1cm}\hcovectors=[t_1,t_2]\cap \hcovectors=\{ t_1,t_2\},$ say $t_2\prec t=t_1.$ Since $t_1$ and $t_2$ are not adjacent in the graph $G(\{ \pm 1,0\}^E)$, there exist at least two distinct coordinates $e$ and $f$ at which they differ, whence $t_2(e)=t_2(f)=0$ and $t_1(e),t_1(f)\in \{ \pm 1\}.$ Then the map $t\in \{ \pm 1, 0\}^E$ defined by
$$
t(g):=\begin{cases}
0 &\mbox{ if } g=e,\\
t_1(g) &\mbox{ if } g\ne e
\end{cases}
$$
is different from $t_1$ and $t_2$ but belongs to $[t_1,t_2]\cap \uparrow\hspace*{-0.1cm}\hcovectors=[t_1,t_2]\cap \hcovectors,$ yielding a final contradiction. This establishes that  $G(\hcovectors)$ is an isometric subgraph of $G(\{ \pm 1, 0\}^E)$.
\end{proof}

We conclude this section by showing that 0-convexity of subsets of $\{ \pm 1, 0\}^E$ is equivalent to 0-convexity of their upward closures:

\begin{lemma} \label{SCA-upper-closure} A subset $\hcovectors$ of $\{ \pm 1, 0\}^E$ is 0-convex if and only if its upward closure  $\uparrow\hspace*{-0.1cm}\hcovectors$ is 0-convex.
\end{lemma}

\begin{proof} Trivially, $\hcovectors$ and $\uparrow\hspace*{-0.1cm}\hcovectors$ have the same set of minimal elements: $\Min(\hcovectors)=\Min(\uparrow\hspace*{-0.1cm}\hcovectors)$.
Assume that $\hcovectors$ is 0-convex and let $t'_1,t'_2\in \uparrow\hspace*{-0.1cm}\hcovectors$
with $t_1\prec t'_1$ and $t_2\prec t'_2$ for some $t_1,t_2\in \hcovectors.$ If $t'_1(e)\cdot t'_2(e)=-1$ for some $e\in E,$ then $t_0$ can be chosen to be one of
$t_1,t_2$ in case that $0\in \{ t_1(e),t_2(e)\}.$ Otherwise, one may choose $t_0$ in $\hcovectors$ with $t_0(e)=0$ and $t_0(f)\in \{ 0,t_1(f),t_2(f)\}$ for all $f\in E.$
Conversely, assume that $\uparrow\hspace*{-0.1cm}\hcovectors$ is 0-convex. Since the minimal choices relative to $\prec$ for establishing 0-convexity
in $\uparrow\hspace*{-0.1cm}\hcovectors$ belong to $\hcovectors$, it follows that 0-convexity holds for $\hcovectors$ as well.
\end{proof}

The analogue of Lemma \ref{SCA-upper-closure} does not hold for sign-convexity:  $\{ \pm 1,0\}^E$ is the upper closure of any subset $\hcovectors$ containing $(0,\ldots,0)$, however $\hcovectors$ is not necessarily sign-convex.


\section{Intrinsic path metrics}\label{path_intrinsic}
In this section we recall some basic notions about intrinsic path metrics and the length of paths, which are relevant for the intrinsic metrics of cubihedra.

In an  arbitrary metric space $(\MM,d)$ one can trivially define the {\it length $\ell (P)$ of a finite $x,y$-path}
$P$ of points  $x =: t_0, t_1,\ldots, t_k := y$ as the sum
$\ell(P):=\sum_{i=1}^k d(t_{i-1},t_i).$
The {\it modulus} of this path $P$ is the maximum of all single step lengths $d(t_{i-1},t_i)$ for $i=1,...,k.$ Then the infimum
$$\mbox{inf} \{ \ell(P): ~P \mbox{ is a finite } x,y\mbox{-path in } \MM \mbox{ with modulus}(P) <\epsilon\}$$
of the lengths of all finite paths between two points $x$ and $y$ in $\MM$ having modulus smaller than $\epsilon$ exists because the
length of every finite path between $x$ and $y$ is bounded below by $d(x,y)$ by virtue of the triangle inequality. Then the supremum
$$\delta_{\MM,d}(x,y):=\mbox{sup}\{\mbox{inf}\{ \ell(P): ~P \mbox{ is a finite } x,y\mbox{-path in } \MM \mbox{ with modulus}(P)<\epsilon\}: ~\epsilon\in\mathbb{R}^+\}$$
will be called the {\it (intrinsic) finite-path distance} between $x$ and $y$ in the metric space $(\MM,d)$, which may formally take
the value $\infty$ of the extended real line when the supremum does not exist. If, however, all values
are real numbers then we could speak of the {\it (intrinsic) finite-path metric} of $(\MM,d)$, since evidently $(\MM,d)$ satisfies the triangle inequality.

If there exists a (general) $x,y$-{\it path} in $(\MM,d)$, that is, a continuous map $\gamma: [0,1] \rightarrow \MM$ with $\gamma(0)=x$ and $\gamma(1)=y$, then its length \cite[p.11]{Pa}
$$L(\gamma):=\sup \Big\{ \sum \limits ^k_{i=1}d(\gamma(t_{i-1}),\gamma(t_i)):
  0\le t_0\le t_1\le \ldots \le t_k\le 1 \mbox{ where } k\ge 1\Big\}$$
is an upper bound for the lengths $\ell(P)$ of all finite $x,y$-paths $P$ contained in $\gamma$. If an $x,y$-path $\gamma$ exists with $L(\gamma)<\infty$, then $x$ and $y$ are said to be {\it connected by a rectifiable path} in $(\MM,d)$ \cite[p.35]{Pa}. Taking the infimum over all general $x,y$-paths, we obtain
$$\delta^{\MM,d}(x,y):=\mbox{inf}\{ L(\gamma): ~\gamma \mbox{ is an } x,y\mbox{-path in} ~\MM\},$$
referred to as the {\it path} (or {\it length}) {\it distance} in $\MM$ relative to $d$ (in the book \cite{Pa} the length distance is denoted by $d_{\ell}$).
If each pair of points in $\MM$ is connected by a rectifiable path, then by \cite[Proposition 2.1.5]{Pa} $\delta^{\MM,d}$ is a metric on $\MM$, referred to
the {\it (intrinsic) path metric} of $\MM$ associated to $d.$ We thus have the inequalities
$$d(x,y)\le \delta_{\MM,d}(x,y)\le \delta^{\MM,d}(x,y),$$
with the understanding that either $\delta$ value could equal $\infty$. In the extreme case, $\delta_{\MM,d}$  could be a metric whereas
$\delta^{\MM,d}$ is the constant $\infty$ map on pairs of distinct points. Consider, for example, the rational unit interval $M =\mathbb{Q}\mbox [0,1]$
with its natural metric $d$: since paths of any modulus exist,
the intrinsic path metric coincides with the natural metric, but geodesics between distinct points do not exist, whence the intrinsic length distance is infinite.
Therefore the advantage of using $\delta_{\MM,d}$ rather than $\delta^{\MM,d}$  is that we do not need to impose existence of rectifiable paths. In the
important case that $(\MM,d)$ is a {\it length} (or  {\it  path-metric) space}  \cite[p.35]{Pa}, that is, $d$ coincides with $\delta^{\MM,d}$,
both concepts of intrinsic metrics can be equated with the original metric $d.$ The property $\delta_{\MM,d}=d$ carries over to dense subspaces of
$(\MM,d)$.  In an arbitrary length space $(\MM,d)$ general $x,y$-paths of length $d(x,y)$, that is, $x,y$-geodesics, need not exist. When however
they exist for all pairs of points then the space is called {\it geodesic}. A geodesic space $(\MM,d)$ is called a {\it real tree} if every
geodesic constitutes the unique path between its end points. Note that a compact real tree $({\MM},d)$ may have infinitely many
{\it branching points}, i.e. points $p$ for which ${\MM}- \{ p\}$ has at least three connected components.

Every finite graph $G=(V,E)$ has a trivial geometric realization as a  1-dimensional cell complex obtained by replacing each edge $\{ x,y\}$ by a solid link, that is, a
copy $[u_x,u_y]$ of the unit interval $[0,1]$ of the real line such that two copies intersect in an endpoint exactly when the corresponding edges are incident. The
resulting geometric graph (alias network) $\MM$ inherits its metric $\delta$ from the standard-graph metric $d$ of $G$ and the usual distance of $[0,1].$ Informally speaking,
the distance  $\delta(p,q)$ between a point $p$ from a link $[u_v,u_w]$ and a point $q$ from a link $[u_x,u_y]$ in the geometric graph $\MM$ is the smallest of the four sums
$|p-u_v|+d(v,x)+|q-u_x|, |p-u_v|+d(v,y)+|q-u_y|,|p-u_w|+d(w,x)+|q-u_x|,$  and  $|p-u_w|+d(w,y)+|q-u_y|.$ More formally, one can apply the general construction of the
length metric as described in \cite[Example 2.1.3(iv)]{Pa} or \cite[Section I.1.9]{BrHa}. In this way, one obtains a geodesic
space $(\MM,\delta).$ Similarly, in ad hoc manner, one could deal with the geometric realization of
any finite cubical complex, but for lopsided sets we can obtain the length space property for the associated geometric realizations in a canonical way; see below.

Completeness is an essential prerequisite in Menger's proof that complete Menger-convex spaces are geodesic. Indeed, the intersection
$(\MM,d)$ of the Cantor set with the irrational (open) unit interval is Menger-convex but $\delta_{\MM,d}(x,y)=\infty$ holds for all pairs $x,y$ of
distinct points. In the absence of Menger-convexity the requirement $\delta_{\MM,d}<\infty$ together with some local compactness condition can at
least guarantee that $(\MM,\delta_{\MM,d})$ is a geodesic space. We say that a metric space is {\it boundedly compact} if every closed bounded
subset is compact. Note that in the context of Riemannian geometry and length spaces one usually calls these spaces ``proper''.

\begin{lemma} \label{delta} If a boundedly compact metric space $(\MM,d)$ admits a finite-path metric $\delta_{\MM,d}<\infty$, then $(\MM,\delta_{\MM,d})$
is a geodesic space, whence $\delta_{\MM,d}=\delta^{\MM,d}$ holds.
\end{lemma}

\begin{proof} For any distinct points $x,y$ we can approximate $\alpha:=\delta_{\MM,d}(x,y)$ by finite
$x,y$-paths of moduli $1/n$  $(n\rightarrow\infty)$. Specifically, for every $n\in \mathbb N$ there exists a finite $x,y$-path $P_n$ of length smaller
than $\alpha(1 + 1/(2n))$ and modulus smaller than $\alpha/(2n)$. On each path $P_n$ pick the first point $y_n$ for which the initial
$x,y_n$-subpath $P'_{n}$ exceeds length $\alpha/2$, whence the length of the final $y_n,y$-subpath of $P_n$ is then smaller than
$\alpha(1 + 1/(2n)).$ Since the sequence $(y_n)$ is contained in the closed $2\alpha$-ball centered at $x$, it includes a
subsequence $(z_n)$ converging to some point $z$ such that, say, $d(z_n,z)<\alpha/(2n)$. Inserting this limit point $z$ into each $P_n$
directly after $y_n$ yields a path consisting of an initial $x,z$-subpath plus a final $z,y$-subpath both of modulus smaller than $\alpha/(2n)$
and with lengths between  $\alpha/2-1/(2n)$ and $\alpha/2+1/n.$ Therefore
$$\delta_{\MM,d}(x,z)+\delta_{\MM,d}(z,y)\le \alpha/2+ \alpha/2=\alpha=\delta_{\MM,d}(x,y),$$
whence by virtue of the triangle inequality $z$ is the midpoint of the segment between $x$ and $y$ relative to the finite path metric $\delta_{\MM,d}$.
We can now iterate this midpoint construction by applying the procedure first to the pairs $x,z$ and $z,y,$ and so on. Then we eventually obtain a
dense subset $Z$ of the segment $[x,y]$ relative to the metric $\delta_{\MM,d}$, admitting an isometry $\zeta$ from $(Z,\delta_{\MM,d}|_{Z\times Z})$
to the set of $\alpha$-multiples of the dyadic fractions of 1 endowed with the natural metric.

Every non-dyadic number $\tau$ from the unit interval is the limit of some sequence $(\tau_n)$ of dyadic fractions of 1, for which the $\alpha$-multiples
each have a pre-image $u_n$ under $\zeta$ in $Z$. Then $(u_n)$ is a Cauchy sequence relative to $\delta_{\MM,d}$ and hence to $d$, which converges to some
point $u$ due to completeness of $(\MM,d)$. For any two points $v$ and $w$ of $Z$ with $\zeta(v)<\alpha\tau$ and $\zeta(w)>\alpha\tau$ almost all $u_n$
are between $v$ and $w$ relative to $\delta_{\MM,d}$. Therefore one can approximate $\delta_{\MM,d}(v,w)$ by pairs of concatenated finite $v,u_n$-paths
and $u_n,w$-paths of moduli converging to 0 and total lengths converging to $\delta_{\MM,d}(v,w)$ $(n\rightarrow\infty)$. Substituting $u_n$ by $u$ in
all paths yields a sequence of concatenated paths approximating both $\delta_{\MM,d}(v,u) +\delta_{\MM,d}(u,w)$ and $\delta_{\MM,d}(v,w)$. This shows
that $u$ is between $v$ and $w$ relative to $\delta_{\MM,d}$. In summary, we have thus established an isometry from the closure $\overline{Z}$ of
$Z$ in $(\MM,\delta_{\MM,d})$ to the interval $[0,\alpha]$, where $\overline{Z}$ is a $x,y$-geodesic in $(\MM,\delta_{\MM,d}).$
\end{proof}

Let $X$ and $Y$ be two sets of maps from some disjoint nonempty sets $A$ and $B$, respectively, to  a set $\Lambda$. Then we write $X\times Y$ for
the set of all maps $r:A\cup B\rightarrow \Lambda$ for which $r|_A$ belongs to $X$ and $r|_B$ belongs to $Y.$ For singletons $X$ or $Y$, set brackets
are omitted. Given a finite set $E$ and a nonempty subset $A$ of $E,$ the set $\Lambda^A\times r_0$ for any $r_0\in \Lambda^{E-A}$ is called a {\it fiber}
of the Cartesian power $\Lambda^E,$ namely the $A$-{\it fiber} of $\Lambda^E$ {\it at} $q_0\times r_0$ for any $q_0\in \Lambda^A.$ If $\Lambda$ is endowed
with a (natural) metric, then we will denote by $d$ the product metric $d^E$ on the the product space $\Lambda^E$, where $E$ is a finite set. When $\Lambda$
is connected, a subspace $K$ of $\Lambda^E$ is called {\it fiber-connected} if the intersection of $K$ with each fiber of $\Lambda^E$ is
connected (or empty). Similarly, when $\Lambda$ is a geodesic space, $K\subseteq \Lambda^E$ is said to be {\it fiber-geodesic} if $K$
intersects each fiber of $\Lambda^E$ in a geodesic subspace (or the empty set). The $E$-fiber of $\Lambda^E$ is understood to be the entire space $\Lambda^E$.

\begin{proposition} \label{fiber-isometric} Let $\Lambda$ be a boundedly compact geodesic space. For a closed subset $K$ of some finite power ${\Lambda}^E$
of $\Lambda$ (endowed with the product metric $d=d^E$), the following statements are equivalent:

\begin{itemize}
\item[\rm(i)] the intrinsic finite-path metric $\delta_{{K},d}$ of $K$ coincides with the restriction $d|_{K}$
of the product distance $d=d^E$ of ${\Lambda}^E$;
\item[\rm(ii)] ${K}$ is a geodesic subspace of ${\Lambda}^E$;
\item[\rm(iii)] ${K}$ is fiber-geodesic.
\end{itemize}
\end{proposition}

\begin{proof} Clearly, (ii)$\Rightarrow$(i)\&(iii)  and (iii)$\Rightarrow$(ii) both hold. Since $\Lambda$ is boundedly compact and
$E$ is finite, ${\Lambda}^E$ is also boundedly compact. This carries over to $({K},d|_{K})$ because $K$ is a closed subset
of ${\Lambda}^E$.   If (i) holds, then $(K,\delta_{{K},d|_{K}})$ is a geodesic
space by Lemma \ref{delta}, because  $\delta_{{K},d|_{K}}=d|_{K}<\infty$.  
Therefore $K$ is a geodesic subspace of ${\Lambda}^E$, thus establishing the implication (i)$\Rightarrow$(ii).
\end{proof}

 Proposition \ref{fiber-isometric} can be applied, for instance, to the closed subsets $K$ of the Euclidean space ${\mathbb R}^E:$ thus, the restriction
 of the Euclidean ($\ell_2$-)metric to $K$ constitutes the intrinsic (finite-) path metric of $K$ relative to the Euclidean metric exactly
 when $K$ is closed under taking line segments, that is, $K$ is convex. In the $\ell_1$ case condition (iii) of Proposition \ref{fiber-isometric}
 can be weakened further depending on the space $\Lambda$. Namely, in order to replace fiber-geodesity by fiber-connectedness, the factors need to be real trees.

\begin{proposition} \label{fiber-connected} Let  $\Lambda$ be a boundedly compact real tree (equipped with the natural metric) in which every segment includes only
finitely many branching points. Then a closed subset $K$ of some finite power ${\Lambda}^E$ of $\Lambda$ is a geodesic subspace of $\Lambda^E$ if and only
if $K$ if fiber-connected.
\end{proposition}

\begin{proof} By Proposition \ref{fiber-isometric}, any closed geodesic subspace of $\Lambda^E$ is fiber-geodesic and hence fiber-connected. Conversely
assume that $K$ is a fiber-connected closed subspace of $\Lambda^E$.  Since $({\Lambda}^E,d)$ is  complete and $K$ is closed, $(K,d|_{K})$ is also
a complete metric space. Thus to prove that $K$ is a geodesic subspace of ${\Lambda}^E$, it suffices
to establish that $K$ is Menger-convex with respect to $d$. We proceed by induction on $\#E.$ Since any proper fiber ${K}_0$ of $K$
is a closed fiber-connected subset of ${\Lambda}^A\times r_0$ for a proper subset $A$ of $E$ and a map $r_0\in \Lambda^{E-A},$
by induction assumption we can assume that ${K}_0$ is a geodesic subspace of ${\Lambda}^A\times \{ r_0\}$ and $\Lambda^E$. Suppose by way
of contradiction that $K$ is not Menger-convex. In view of the induction hypothesis, we may then suppose that
there exist $r_1, r_2 \in {K}$ with
$$[r_1, r_2]_{K}=K \cap [r_1,r_2]_{\Lambda^E}= \{ r_1, r_2\}$$
such that $r_1(x) \neq r_2(x)$ holds for all $x \in E.$

For any $x\in E$, let $\Lambda^x$ be the $x$th factor (a copy of $\Lambda$) of the product ${\Lambda}^E$. Then $\Lambda^{E- \{ x\}}\times r_1(x)$ is the
$(E- \{ x\})$-fiber of $\Lambda^E$ that contains the point $r_1$ of $K$. By $\Lambda^x_+=\Lambda^x_+(r_1,r_2)$  denote the set of all
points $\lambda$ of the real tree
$\Lambda^x$ for which $r_1(x)$ is not between $\lambda$ and $r_2(x),$ that is,
$$\Lambda^x_+:=\{ \lambda\in \Lambda^x: ~ [\lambda,r_1(x)]_{\Lambda^E}\cap [r_1(x),r_2(x)]_{\Lambda^E}\ne \{ r_1(x)\}\}.$$
Then $\Lambda^x_+$ is an open subset of the real tree $\Lambda^x$ and its closure $\overline{\Lambda^x_+}=\Lambda^x_+\cup\{ r_1(x)\}$ is a boundedly
compact subtree of $\Lambda^x$. Trivially, the closure of $\Lambda^{E}_+:=\prod_{x\in E}\Lambda^x_+$ equals
$$\overline{\Lambda^E_+}=\prod_{x\in E} (\Lambda_+^x\cup \{ r_1(x)\}).$$
(Note that $\Lambda^E_+$ resp. $\overline{\Lambda^E_+}$ equal the open resp. closed first orthant of ${\mathbb R}^E$ in the case that $\Lambda={\mathbb R},$
$r_1={\mathbf 0}$ and $r_2>{\mathbf 0}$.)  We claim that
$${K}\cap \overline{\Lambda^{E}_+}=({K}\cap \Lambda^{E}_+)\cup \{ r_1(x)\}.$$
Suppose the contrary, that is, suppose that there exists a point $$r\in ({K}- \{ r_1\})\cap \prod_{y\in E}(\Lambda^y_+\cup \{ r_1(y)\} \mbox{ with } r(x)=r_1(x)$$
for some $x\in E$. Then both $r$ and $r_1$ belong to the fiber $\Lambda^{E- \{ x\}}\times r_1(x)$ and hence are connected by a geodesic $\gamma$ in $K$, according to the induction hypothesis. On the other hand, as $\Lambda$ is a real tree, there exists a point $s$ of $\Lambda^E$ (the median point of $r_1,r,s$) such that
$$[r_1,s]_{\Lambda^E}=[r_1,r_2]_{\Lambda^E}\cap [r_1,r]_{\Lambda^E}.$$
The complement $A=\{ y\in E: ~ r_1(y)\ne r(y)\}$ of the equalizer of $r_1$ and $r$ does not contain the element $x$. Then the geodesic $\gamma$ between $r_1$
and $r$ is included in the $A$-fiber at $r_1$ (and $r$). By the initial hypothesis, $r(y)\in \Lambda^y_+$ and hence $r_1(y)\ne s(y)$ for each $y\in A$. Let
$\epsilon$ be the minimum of the distances of $r_1(y)$ and $s(y)$ for $y\in A$ in the copies of the real tree $\Lambda$. Then the intersection of $\gamma$
with the closed ball of radius $\epsilon$ centered at $r_1$ is included in the box
$$[r_1|_{A},s|_{A}]_{\Lambda^{A}}\times r_1|_{E- A}\subseteq [r_1,r_2]_{\Lambda^E},$$
whence
$$\{ r_1\}\varsubsetneqq\gamma\cap [r_1,r_2]_{\Lambda^E}\subseteq [r_1,r_2]_{K},$$
contrary to the assumption that $[r_1,r_2]_{K}=\{ r_1,r_2\}.$

The intersection ${K}_+:={K}\cap {\Lambda}^E_+$ is an open set in the (topological) subspace $K$ of $\Lambda^E$ that includes $r_2$ but not $r_1$.
We wish to show that ${K}_+$ is closed as well. By what has just been shown,
$$\overline{{K}_+}\subseteq {K}_+\cup \{ r_1\}\subseteq {K}.$$
So suppose that $r_1\in \overline{K_+}.$ Then there exists a sequence $(s_n)$ in $K_+$ converging to $r_1.$ The sequence $(t_n)$ of median points $t_n$ of
$r_1,r_2,$ and $s_n$ $(n\in {\mathbb N})$ also converges to $r_1.$ Since $(s_n)$ is fully included in ${\Lambda}^E_+,$ so is $(t_n).$ It follows from $[r_1,r_2]_{K}=\{ r_1,r_2\}$
that either $t_n(x)=r_2(x)$ or $t_n(x)$ is a branching point of $\Lambda^x$ for each coordinate $x$ and index $n$. This, however, conflicts with the initial requirement
on $\Lambda$ that all segments of $\Lambda$ contain only finitely many branching points. We conclude that $\overline{K_+}={K}_+$ as asserted. Therefore $K_+$ is
both open and closed in $K$, which finally contradicts connectedness of $K$, and the proof is complete.
\end{proof}

A standard example shows that the finiteness condition on branching points in segments cannot be dropped in Proposition \ref{fiber-connected}. Consider the
compact tree $\Lambda$ shown in Figure \ref{figure1} in two copies $\Lambda^x$ and $\Lambda^y$. It connects three vertices of a right-angled triangle with
legs of length 1 where the sequence of leaves is located on the hypotenuse and  converges to the intersection point $\mu$ of the hypothenuse and one leg.
The branching points of $\Lambda$ lie on that leg and converge also to $\mu$.  The total length of $\Lambda$ equals 3. A path
$\gamma=\gamma_0\cup \gamma_1\cup \gamma_2\cup\ldots\cup \{ (\mu^x,\mu^y)\}$ between two points $(\lambda^x,\lambda^y)$ and $(\mu^x,\mu^y)$ in $\Lambda^x\times\Lambda^y$
can be constructed by alternating $x$-fibers $\gamma_0,\gamma_2,\ldots$ and $y$-fibers $\gamma_1,\gamma_3,\ldots,$ with $(\mu^x,\mu^y)$ as limiting point; see Figure 1
for the two projections of $\gamma$ on $\Lambda^x$ and $\Lambda^y.$ Then $\gamma$ intersects the segment between $(\lambda^x,\lambda^y)$ and $(\mu^x,\mu^y)$ only in its
two end points. Thus, the closed subset $\gamma$ of $\Lambda^x\times\Lambda^y$ is fiber-connected but not Menger-convex and hence does not constitute a geodesic subspace.

\begin{figure}[t]
\vspace*{-3cm}
\centering\includegraphics[scale=0.45]{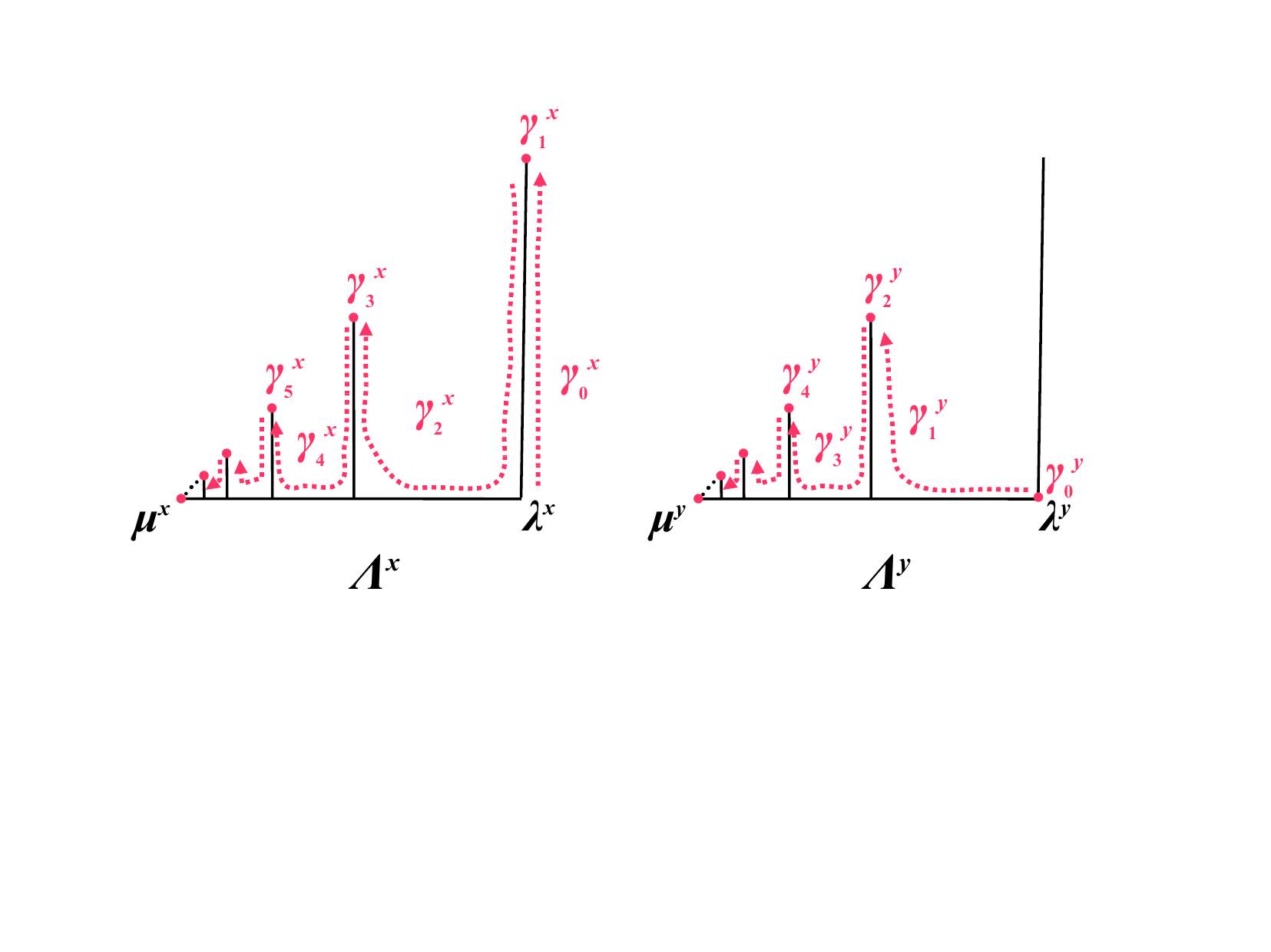}
\vspace*{-3cm}
\caption{Projection of a geodesic $\gamma$ in $\Lambda^x\times\Lambda^y$ to the factors}
\label{figure1}
\end{figure}

\section{Weak and sign convexity of cubihedra}\label{metric-characterization}

In this section we will characterize the ample sets via their cubihedra. We show that they are exactly the subsets of $\{ \pm 1\}^E$ whose cubihedra are weakly convex or sign-convex. We will also characterize the subsets of $\{\pm 1, 0\}^E$ whose upward closures are the barycentric completions of ample sets as. Namely, we show that they are exactly  the 0-convex subsets of $\{\pm 1, 0\}^E$. We also provide other metric characterizations of cubihedra  and barycentric completions of ample sets.

For a connected subset $\covectors$ of $\{ \pm 1\}^E$, its geometric realization $|\covectors|$ within the hypercube ${\HH}(E)$
(endowed with the $\ell_1$-metric $d$) admits an intrinsic path metric $d_{|\covectors|}=\delta^{|\covectors|,d}$.  Indeed, any pair of points in $|\covectors|$ can be connected by a rectifiable path in $|\covectors|$ relative to $d,$ whence $d_{|\covectors|}$ exists by
virtue of Lemma \ref{delta}. Although $(|\covectors|,d_{|\covectors}|)$ is a geodesic space in its own right, it is not necessarily a metric subspace of
$({\mathbb R}^E,d),$ even when $\covectors$ is isometric. 
The first theorem proved in this section shows that a set of the form $K=|\covectors|$ is a metric subspace of $({\mathbb R}^E,d)$ (i.e., it is weakly convex) if and only if $L$ is ample.

\begin{theorem} \label{metric-char} For a subset $\covectors$ of $\{ \pm 1\}^E,$ the following
statements are equivalent:

\begin{itemize}
\item[\rm(i)] ${\covectors}$ is ample;
\item[\rm(ii)] $|{\covectors}|$ is weakly convex;
\item[\rm(iii)]  $|\covectors|$ is a sign-convex;
\item[\rm(iv)]  $\Baryc(\covectors)$ is sign-convex;
\item[\rm(v)] $\covectors$ is isometric  and every face of $|{\covectors}|$ is gated in
$(|\covectors|,d_{|\covectors}|);$
\item[\rm(vi)] the restrictions of the intrinsic metric $d_{|{\covectors}|}$ and the $\ell_1$-metric  $d$ on $\Baryc(\covectors)$ coincide;
\item[\rm(vii)] the barycenters $t_1$ and $t_2$ of any two parallel faces of $|{\covectors}|$ have the same distance with respect to the intrinsic metric and $\ell_1$-metric:
$$d_{|\covectors|}(t_1,t_2)=d(t_1,t_2).$$
\end{itemize}
\end{theorem}

\begin{proof} To prove the equivalence of the conditions from (i) to (vii), we establish three chains of implications $\rm(i)\Rightarrow \rm(ii)\Rightarrow \rm(iii)\Rightarrow \rm(iv)\Rightarrow \rm(vi)\Rightarrow \rm(vii)$,  $\rm(ii)\Rightarrow \rm(v)\Rightarrow \rm(vii)$ and $\rm{(vii)}\Rightarrow \rm(i)$.

To establish (i)$\Rightarrow$(ii), we only need to show that $|\covectors|$ is fiber-connected, by virtue of Propositions  \ref{fiber-isometric} and \ref{fiber-connected}. A fiber $\FF$ of $|\covectors|$ is the intersection of $|\covectors|$ with some $A$-plane, say,
$${\FF}= |\covectors|\cap ({\mathbb R}^A \times r|_{E-A}) \mbox{ for some } r\in |\covectors| \mbox{ and } A\subseteq E.$$
Note that the smallest face of $|\covectors|_A$ containing $r|_{E-A}$ has the form
$$F(r|_{E-A})=\HH(B) \times r|_{E-A-B}, \mbox{ where }$$
$$B=E(r)-A=\{ e\in E-A: ~-1<r(e)<+1\}.$$
Thus, the smallest face of $|\covectors|$ containing any point of $\FF$ has the hypercube $\HH(B)$  as its factor:
$$\HH(B) \times {\FF}_B\subseteq |\covectors|.$$
This can also be expressed by saying that ${\FF}_B=\{ q|_{E-B}: ~q\in {\FF}\}$ is the $A$-fiber of $|\covectors^B|$ containing $r|_{E-B}.$ Since the positions at
which the points of ${\FF}_B$ have their coordinates properly between -1 and +1 all belong to $A,$ we infer that ${\FF}_B$ equals the geometric realization
of the $A$-fiber of $\covectors^B$ at $r|_{E-B}.$ Since $\covectors^B$ and its fibers are ample and hence connected by \cite[Theorem 3]{BaChDrKo}, we conclude that ${\FF}_B$
is connected and so is $\FF$. This shows that  (i)$\Rightarrow$(ii).

The implications  (ii)$\Rightarrow$(iii)$\Rightarrow$(iv) follow  from Lemmas \ref{weakconv->signconv}  and \ref{sign-convex}. Next we show that (iv)$\Rightarrow$(vi) holds. Pick any $t_1,t_2\in \Baryc(\covectors)$. We prove the equality $d_{|L|}(t_1,t_2)=d(t_1,t_2)$ by induction on $d(t_1,t_2)$. First suppose that there exists $e\in E$ such that $t_1(e)t_2(e)=-1$. By Lemma \ref{upper-set-baryc}, $\Baryc(\covectors)$ is upward closed. Therefore, by Proposition \ref{sign-convex-upper-closed}  there exists some $t_0\in [t_1,t_2]\cap \Baryc(\covectors)$ with $t_0(e)=0$ and $t_0(f)\in \{ t_1(f),t_2(f)\}$ for any $f\in E$ such that $t_1(f)t_2(f)\ge 0$. Since $t_0(e)=0$, $t_0$ is different from $t_1$ and $t_2$. Since $t_0\in [t_1,t_2]$, we have $d(t_1,t_0)<d(t_1,t_2)$ and $d(t_0,t_2)<d(t_1,t_2)$. By induction hypothesis, $d_{|L|}(t_1,t_0)=d(t_1,t_0)$ and $d_{|L|}(t_0,t_2)=d(t_0,t_2)$. Therefore $|L|$ contains a path of length $d(t_1,t_0)$ joining $t_1$ and $t_0$ and a path of length $d(t_0,t_2)$ joining $t_0$ and $t_2$. The union of these paths is a path of length $d(t_1,t_2)=d(t_1,t_0)+d(t_0,t_2)$ between $t_1$ and $t_2$, showing that $d_{|L|}(t_1,t_2)=d(t_1,t_2)$. This concludes the proof of the implication  (iv)$\Rightarrow$(vi).  The implication 
(vi)$\Rightarrow$(vii) is trivial.

We continue with the implications $\rm(ii)\Rightarrow \rm(v)\Rightarrow \rm(vii)$. To prove (ii)$\Rightarrow$(v),  for a face $\FF$ of $|\covectors|$ and a point $r\in |\covectors|,$ take the gate $r'$ of $r$ in $\FF$ relative to the $\ell_1$-metric $d$ of the hypercube $\HH(E)$.  Since $r'\in {\FF}\subseteq |\covectors|$ and $|\covectors|$ is path-$\ell_1$-isometric, we have $d_{|\covectors|}(r,r')=d(r,r'),$ whence $r'$ also serves as the corresponding gate within the cubihedron $|\covectors|.$

Next we show that (v)$\Rightarrow$(vii) holds. If $t_1$ and $t_2$ are the barycenters of two parallel faces, then these faces are $F(t_1)$ and $F(t_2)$ having the dimension $k$,
say. Then they lie on parallel $A$-planes for some subset $A\subseteq E$ with $\#A=k.$ Let $q$ be any vertex of $F(t_1)$ (necessarily belonging to $\covectors$) and $r$ be the
corresponding vertex from $F(t_2)$ (and $\covectors$), thus satisfying $q|_A=r|_A.$ Since $\covectors$ is isometric and $t_1$ is the barycenter of $F(t_1)$, we have
\begin{align*} d_{|\covectors|}(t_1,r)&\le d_{|\covectors|}(t_1,q)+d_{|\covectors|}(q,r)\\
&=d(t_1,q)+d(q,r)=d(t_1,r),
\end{align*}
whence equality holds. Therefore the gate of $t_1$ in $F(t_2)$ must have distance $k/2$ to all vertices of $F(t_2)$. The unique point in $F(t_2)$ with this property is
the barycenter $t_2$. Consequently, $t_2$ is the gate of $t_1$ in $F(t_2)$ relative to the intrinsic metric $d_{|\covectors|}$. In particular, $d_{|\covectors|}(t_1,t_2)=d(q,r)=d(t_1,t_2),$ as required. This establishes (v)$\Rightarrow$(vii).

Finally, we establish the implication (vii)$\Rightarrow$(i). Connect the barycenters $t_1$ and $t_2$ of two parallel $A$-faces $F(t_1)$ and $F(t_2)$ by a geodesic $\gamma$ in $|\covectors|.$ Then every point $r$ of $\gamma$ has the same projection on $\HH(A)$  as $t_1$ and $t_2$. Therefore $F(r)$ has a $A$-cube as a factor, and consequently, the geodesic $\gamma$ projects onto a geodesic
$\gamma^A$ of $|\covectors|^A$ which connects the vertices $t_1|_{E-A}$ and $t_2|_{E-A}$ of $\covectors^A$.  To show that $t_1|_{E-A}$ and $t_2|_{E-A}$ are at distance $d(t_1,t_2)$ in $\covectors^A$, we use induction on the $\ell_1$-distance between $t_1|_{E-A}$ and $t_2|_{E-A}.$ Let $r$ be the point of $\gamma$ at distance 2 from $t_1$. Then $F(r)$ is some $(A\cup B)$-cube within $\covectors$ with $\varnothing\ne B\subseteq E-A,$ which necessarily includes $F(t_1)$ as a face. Then every point in $F(r|_{E-A})$ belongs to $|\covectors|^A.$ In particular, the neighbor
$s'$ of $r|_{E-A}$ with $r|_{E-A-\{ e\}}=s'|_{E-A-\{ e\}}$ and $s'(e)=-r(e)$ for some $e\in B$ also belongs to $|\covectors|^A$ and is between $t_1|_{E-A}$ and $t_2|_{E-A}.$ By virtue of the induction hypothesis $s'$ and $t_2|_{E-A}$ are at distance $d(s',t_2|_{E-A})=d(u,v)-2.$ Therefore $\covectors^A$ is isometric and consequently $\covectors$ is ample.
\end{proof}

The second theorem of this section characterizes the subsets $\hcovectors$ of $\{ \pm 1, 0\}^E$ whose upward closures  $\uparrow\hspace*{-0.1cm}\hcovectors$ restricted to $\{ \pm 1\}^E$ are ample. We prove that they are exactly the 0-convex subsets of $\{ \pm 1, 0\}^E$ (i.e., the subsets satisfying (SCA)).

\begin{theorem} \label{barycenters-ample} For a subset $\hcovectors$ of $\{ \pm 1,0\}^E$ and $\covectors:=\uparrow\hspace*{-0.1cm}\hcovectors\cap \{ \pm 1\}^E$, the following statements are equivalent:
\begin{itemize}
\item[\rm(i)] $[\hcovectors]$ is weakly convex;
\item[\rm(ii)] $\uparrow\hspace*{-0.1cm}\hcovectors$ is sign-convex;
\item[\rm(iii)] $\uparrow\hspace*{-0.1cm}\hcovectors$ is 0-convex;
\item[\rm(iv)] $\hcovectors$ is 0-convex;
\item[\rm(v)] $\uparrow\hspace*{-0.1cm}\hcovectors$ is an isometric subset of the grid graph $G(\{ \pm 1,0\}^E)$;
\item[\rm(vi)] $\covectors$ is ample such that $[\hcovectors]=|\covectors|$
and  $\uparrow\hspace*{-0.1cm}\hcovectors=\Baryc(\covectors)$.
\end{itemize}
\end{theorem}

\begin{proof}  If $[\hcovectors]$ is weakly convex, then $[\hcovectors]$
is sign-convex by Lemma \ref{weakconv->signconv}. By Lemma \ref{sign-convex}, the set  $\hcovectors([\hcovectors])=\uparrow\hspace*{-0.1cm}\hcovectors$ is sign-convex. This establishes (i)$\Rightarrow$(ii). The equivalence  (ii)$\Longleftrightarrow$(iii) follows from the first assertion of Proposition \ref{sign-convex-upper-closed} and the equivalence (iii)$\Longleftrightarrow$(iv) follows from Lemma \ref{SCA-upper-closure}.
The implication (ii)$\Rightarrow$(v) follows from the second assertion of Proposition \ref{sign-convex-upper-closed}.

For the proof of (v)$\Rightarrow$(vi) we first claim that the smallest isometric subset $\UU$ of the grid graph $G(\{ \pm 1,0\}^E)$ that includes some $A$-fiber of $\{ \pm 1\}^E$ at some vertex $s\in \{ \pm 1\}^E$ is the $A$-fiber of $\{ \pm 1,0\}^E$ at $s$. In fact, as $\{ \pm 1\}^A\times s|_{E-A}\subseteq {\UU}$ and in
each coordinate 0 is needed to connect $-1$ to $+1$, we may assume by induction on $\# A$ that for some $e\in A,$
$$\{ \pm 1,0\}^{A-\{ e\}}\times \{ \pm 1\}^e\times s|_{E-A}\subseteq {\UU},$$
whence $\{ \pm 1,0\}^{A-\{ e\}}\times \{ 0\}^e\times s|_{E-A}$ constitutes the set of unique common neighbors in the grid graph
for the pairs $t_1,t_2$ with $t_1|_{A-\{ e\}}=t_2|_{A-\{ e\}},$ $\{ t_1(e),t_2(e)\}=\{ \pm 1\}$, and $t_1|_{E-A}=s|_{E-A}=t_2|_{E-A}.$ Hence $\{ \pm 1,0\}^A\times s|_{E-A}\subseteq {\UU},$ as asserted.

Since $\uparrow\hspace*{-0.1cm}\hcovectors$ is isometric in $G(\{ \pm 1,0\}^E),$ we can apply the preceding observation to infer that for $\covectors=\uparrow\hspace*{-0.1cm}\hcovectors\cap \{ \pm 1\}^E$ the upward closure $\uparrow\hspace*{-0.1cm}\hcovectors$ encompasses $\mbox{Baryc}(\covectors).$ Since the reverse inclusion is trivial because $\uparrow\hspace*{-0.1cm}\hcovectors$ is  an upward closed set, we have thus established $\uparrow\hspace*{-0.1cm}\hcovectors=\mbox{Baryc}(\covectors),$ whence it follows that
$$[\hcovectors]=[\uparrow\hspace*{-0.1cm}\hcovectors]=[\mbox{Baryc}(\covectors)]=|\covectors|,$$
as required. Isometry of $\uparrow\hspace*{-0.1cm}\hcovectors$ also entails that ${\covectors}^A$ is isometric for every $A\subseteq E.$ Indeed, for each pair $s_1,s_2\in \covectors^A,$ the barycenter maps $t_i$ corresponding to $\{ \pm 1\}^A\times s_i$ $(i=1,2)$ belong to $\uparrow\hspace*{-0.1cm}\hcovectors$ by what has just been observed. Then any isometric path connecting $t_1$ and $t_2$ in $\uparrow\hspace*{-0.1cm}\hcovectors$ projects to an isometric path connecting $s_1$ and $s_2$ in $G(\{ \pm 1,0\}^{E-A})$ because $t_1|_A=t_2|_A$ is the zero map on $A$ and $t_i|_{E-A}=s_i$ for $i=1,2.$ To show that $s_1$ and $s_2$ are actually connected by a shortest path in $G(\{ \pm 1\}^{E-A})$ (which is scale 2 embedded in $G(\{ \pm 1,0\}^{E-A})$), we use a trivial induction on the $\ell_1$-distance between $s_1$ and $s_2$. Let the neighbor $w$ of $s_1$ on the projected path equal $0$ at the coordinate $e\in E-A.$ Then the sign map $s'_1$ with
$s'_1|_{E-A-\{ e\}}=s_1|_{E-A-\{ e\}}=w|_{E-A-\{ e\}}$ and $s'_1(e)=-s_1(e)=s_2(e)$ belongs to $\covectors^A$ because $w\prec s'_1.$ Since $s'_1$ is a neighbor of $s_1$ between $s_1$ and $s_2$ in $G(\{ \pm 1\}^{E-A}),$ we conclude that $\covectors^A$ is isometric. This finally shows that $\covectors$ is ample, finishing the proof of (v)$\Rightarrow$(vi).

Finally, if (vi) holds, then the geometric realization $|\covectors|$ of the ample set $\covectors=\uparrow\hspace*{-0.1cm}\hcovectors\cap \{ \pm 1\}^E$ is a weakly convex subset of $\HH(E)$ according to Theorem \ref{metric-char} and coincides with $[\hcovectors]$ by (vi). This establishes (i), and the proof is complete.
\end{proof}

\noindent
{\bf Remark 1.}
Conditions $\rm{(ii)}-\rm{(vi)}$ of Theorem \ref{barycenters-ample} are purely combinatorial and
their equivalence can be proven without employing the metric/topological features of the entire geometric realization. In fact, our combinatorial proof established $\rm{(ii)}\Longleftrightarrow\rm{(iii)}\Longleftrightarrow\rm{(iv)}$ and $\rm{(ii)}\Rightarrow\rm{(v)}\Rightarrow\rm{(vi)}$.
The implication left, $\rm{(vi)}\Rightarrow\rm{(iii)}$ can be shown directly,
without the (de)tour through cubihedra: Assuming that $\hcovectors\subseteq \{ \pm 1,0\}^E$ satisfies (vi), the ``top'' set $\covectors:=\uparrow\hspace*{-0.1cm}\hcovectors\cap \{ \pm 1\}^E$ of sign maps is ample, and the second part of (vi) guarantees that all barycentric maps of the cubihedron $|\covectors|$ belong to $\uparrow\hspace*{-0.1cm}\hcovectors$. Let $t_1$ and $t_2$ be two members of
$\uparrow\hspace*{-0.1cm}\hcovectors$ with $t_1(e)=-1$ and $t_2(e)=+1$ for some $e\in E$. The zero coordinates of $t_i$ determine the set $A_i\subsetneqq E,$ so that $t_i$ encodes some $A_i$-cube of $\HH^{(1)}(E)$ for $i=1,2.$ The $A_1$-cube and $A_2$-cube admit mutually nearest vertices (gates) $s_1$ and $s_2$ within $\HH^{(1)}(E)$. Then $t_1\prec s_1$ and $t_2\prec s_2.$ By the choice of $e$ and the gate property for $s_1$ and $s_2,$ we have
$$e\in \Delta (s_1,s_2)=\{ f\in E: ~s_1(f)\ne s_2(f)\}\subseteq E-(A_1\cup A_2).$$
Since $\covectors$ is ample, $\covectors^{A_1\cap A_2}$ is isometric, whence there exists a shortest path $P$ in $\covectors$
connecting $s_1$ and $s_2$, which projects onto a shortest path between $s_1|_{A_1\cap A_2}$ and $s_2|_{A_1\cap A_2}$
in $\covectors^{A_1\cap A_2}.$ Necessarily, $P$ passes through two adjacent vertices $s'_1$ and $s'_2$ with $s'_1(e)=-1$ and $s'_2(e)=+1.$
Thus, there exist $A_1\cap A_2$-cubes at $s'_1$ and $s'_2$, which are fibers of a $(A_1\cap A_2)\cup \{ e\}$-cube containing
$s'_1$ and $s'_2$. Let $t_0$ denote the barycenter map of this latter cube. Then
$$
t_0(f):=\begin{cases}
0 &\mbox{ if } f\in (A_1\cap A_2)\cup \{ e\},\\
s_2(f)=t_2(f) &\mbox{ if } f\in \Delta(s_1,s'_1),\\
s_1(f)=t_1(f) &\mbox{ if } f\in \Delta (s'_1,s_2),\\
s_1(f)=s_2(f) &\mbox{ otherwise. }
\end{cases}
$$
Note that for $g\in A_i-A_j$ one has $s_i(g)=t_0(g)=t_j(g),$ where $\{ i,j\}=\{ 1,2\}.$ Therefore $t_0$ satisfies the requirements in the definition of (SCA), establishing (vi)$\Rightarrow$(iii).

\section{Projection and dimension}\label{projection}

We will now show that for a subset $\covectors$ of $\{ \pm 1\}^E$ the operators of (orthogonal) projection and geometric
realization commute exactly when $\covectors$ is ample. To this end we will make use of the calculus involving sets of the form $({\covectors}^B)_A$, as developed
in \cite{BaChDrKo}. First observe that, essentially by definition, we have
\begin{align*}
|{\covectors}|=\bigcup \limits _{B\subseteq E}\bigcup \limits
 _{s\in {\covectors}^B}\HH(B) \times s. \tag{1}
\end{align*}
Therefore the topological dimension of the cubihedron $|\covectors|$ can be expressed as
$$\dim |{\covectors}|=\max \,\{ \#B: {\covectors}^B\ne
\varnothing \}=\max \,\{ \# B: B\in \underline{{\X}}({\covectors})\},$$
using the terminology of \cite{BaChDrKo}; see also the Introduction. Applying the above equation to ${\covectors}_A$ instead of $\covectors$ yields
\begin{align*}
 |{\covectors}_A|=\bigcup \limits _{B\subseteq E-A}\bigcup \limits _{t\in
 ({\covectors}_A)^B}\HH(B)\times t. \tag{2}
\end{align*}
For the projection from $\HH(E)$ to ${\HH}(E-A)$ applied to $K=|\covectors|$ we compute
\begin{align*}
|{\covectors}|_A&=\left(\bigcup_{B\subseteq E}\bigcup_{s\in {\covectors}^B} \HH(B)\times s\right)_A\\
&=\bigcup_{B\subseteq E}\bigcup_{s\in {\covectors}^B}\HH(B-A)\times s|_{(E-B)-A}\\
&=\bigcup_{B\subseteq E}\bigcup_{t\in ({\covectors}^B)_{A-B}}\HH(B-A)\times t\\
&\subseteq \bigcup_{B\subseteq E}\bigcup_{t\in (\covectors^{B-A})_A}\HH(B-A)\times t
\end{align*}
because for every subset $B$ of $E$ we have the inclusion
$$({\covectors}^B)_{A - B}=(({\covectors}^{B - A})^{B\cap A})_{A-
  B}\subseteq (({\covectors}^{B- A})_{B\cap A})_{A- B}=({\covectors}^{B - A})_A.$$
Note that here the first and last equations use formulas from (13) of \cite{BaChDrKo}, whereas the inclusion is
derived from formula (12) of \cite{BaChDrKo}. If $B$ is chosen to be disjoint from $A$, then the preceding chain
of expressions just collapses to $(\covectors^B)_A$ and $\HH(B)$ equals $\HH(B-A).$ This shows that actually
equality holds above:
\begin{align*}
|{\covectors}|_A=\bigcup_{B\subseteq E-A}\bigcup_{t\in (\covectors^B)_A} \HH(B)\times t.\tag{3}
\end{align*}
Since $(\covectors^B)_A\subseteq (\covectors_A)^B$ by (13) of \cite{BaChDrKo}, we infer from (2) and (3) the inclusion
$$|\covectors|_A\subseteq |\covectors_A|.$$
Using (3), the dimension of the projected cubihedron can be expressed as
\begin{align*}
\dim |\covectors|_A&=\max \{ \#B: ~(\covectors^B)_A\ne\varnothing \mbox{ for some } B\subseteq E-A\}  \tag{4}\\
&=\max\{ \#B: ~B\subseteq E-A \mbox{ with } \covectors^B\ne\varnothing\}.
\end{align*}

Now, the prerequisites for proving the announced result are all in place.

\begin{theorem}
For a subset  ${\covectors}\subseteq \{ \pm 1\}^E,$
the following statements are equivalent:
\begin{itemize}
\item[\rm(i)] ${\covectors}$ is ample;
\item[\rm(ii)] $|{\covectors}|_A=|{\covectors}_A|$ holds for all $A\subseteq E$;
\item[\rm(iii)] $\dim |{\covectors}|_A=\dim |{\covectors}_A|$ holds for
  all $A\subseteq E$.
\end{itemize}
\end{theorem}

\begin{proof} If $\covectors$ is ample, then $(\covectors^B)_A=(\covectors_A)^B$ according to \cite[Theorem 2]{BaChDrKo} and consequently
$|\covectors|_A$ and $|\covectors_A|$ are equal by (2) and (3). The latter equality trivially implies equality of the corresponding dimensions.

To complete the proof assume that
\begin{align*}
\dim |\covectors|_{E-A}=\dim |\covectors_{E-A}| \mbox{ for all } A\in \overline{\X}(\covectors),
\end{align*}
that is, $A\subseteq E$ with $\covectors_{E-A}=\{ \pm 1\}^A.$ Then
\begin{align*}
\# A &=\dim \HH(A)=\dim |{\covectors}_{E-A}|\\
&=\dim|{\covectors}|_{E-A}\\
&= \max\{ \# B: ~B\subseteq A \mbox{ with } \covectors^B\ne\varnothing\},
\end{align*}
whence $\covectors^A\ne\varnothing,$ that is, $A\in \underline{\X}(\covectors).$ Therefore $\covectors$ is ample by \cite[Theorem 2]{BaChDrKo}.
\end{proof}

\section{Orthant intersection pattern}\label{realization}

Given any
subset $K$ of ${\mathbb R}^E,$ the set
\begin{align*}
L(K)=\{ s\in \{ \pm 1\}^E: ~K\cap \OO(s)\ne\varnothing\}
\end{align*}
encodes the intersection pattern of $K$ with the closed orthants of ${\mathbb R}^E$
determined by the sign maps of $\{ \pm 1\}^E$.
This construction was performed for convex sets by Lawrence \cite{La}. He showed that $L(K)$ is ample whenever $K$ is  a convex set
in the Euclidean space ${\mathbb R}^E$, but not every ample set $\covectors$ can be realized in this way.

The clue for a realization within a wider class
of subsets of ${\mathbb R}^E$ with some weaker convexity properties comes from the rather obvious realization
\begin{align*}
L(|\covectors|)=\covectors \mbox{ if } \covectors\subseteq \{ \pm 1\}^E \mbox{ is ample}. \tag{5}
\end{align*}
Indeed, the inclusion $\covectors\subseteq L(|\covectors|)$ is trivial. Now, if $t\in {\covectors}^*=\{ \pm 1\}^E-\covectors,$ then the corresponding
closed orthant $\OO(t)$ does not intersect any fiber $\FF\subseteq \covectors$ of $\{ \pm 1\}^E$ and hence is disjoint from the face ${\FF}$ of $|\covectors|,$
whence $t\notin |\covectors|.$

We will now show that for a closed subset $K$ of ${\mathbb R}^E$ weak convexity suffices to ensure that $L(K)$ is ample. To
this end we may consider only compact subsets of $H(E)$. Since $L(K)$ is finite, there exists a finite subset $K_0$ of
$K$ with $L(K_0)=L(K).$ We may scale $K$ with some $\lambda>0$ such that $\lambda K_0\subseteq \HH(E).$ Then
\begin{align*}
\widetilde{K}:=\HH(E)\cap \lambda K
\end{align*}
is a compact weakly convex subset of  $\HH(E)$ with $L(\widetilde{K})=L(K).$

\begin{theorem} \label{realizability} A subset ${\covectors}$ of $\{ \pm 1\}^E$ is ample
if and only if there exists a weakly convex subset $K$
of $\mathbb{R}^E$  (or, equivalently, a compact weakly convex subset $K$ of $\HH(E)$)
with $L=L(K)$.
\end{theorem}

\begin{proof} 
It remains to show that $\covectors:=L(K)$ is ample whenever $K$ is a closed weakly convex subset of $\HH(E).$ We proceed by induction on $\#E.$ Since $K$ is weakly convex, $K$ is path-$\ell_1$-isometric, thus there exists $e\in E$ with $\covectors^e\ne\varnothing$.  We claim that $\covectors^e=L(K')$, where
$$K'=\{ r|_{E-\{ e\}}: ~r\in K \mbox{ with } r(e)=0\}.$$
Clearly, $\{ \pm 1\}\times L(K')\subseteq \covectors$. Conversely, if $s'\in \covectors^e,$ then both extensions $s_1,s_2\in \{ \pm 1\}^E$ of $s'$ (with $s_1(e)=-1$ and
$s_2(e)=+1$) belong to $\covectors$. Hence there exist $r_1,r_2\in K$ with $r_1(e)=-1, r_2(e)=+1$ and $r_1(x)\cdot s'(x)\ge 0,$ $r_2(x)\cdot s'(x)\ge 0$ for all $x\in E-\{ e\}.$
Then any geodesic $\gamma$ connecting $r_1$ and $r_2$ in $K$ must contain a point $r$ with $r(e)=0,$ so that necessarily $r(x)\cdot s'(x)\ge 0$ for all $x\in E$ holds as well. This establishes $\covectors^e=L(K').$ Obviously,  the intersection of $K$ with the $e$-hyperplane through $0$ is path-$\ell_1$-isometric, and thus weakly convex.  Therefore $\covectors^e$
is ample by the induction hypothesis.

Finally, to prove that $\covectors$ is isometric, assume that for some subset $A$ with $\#A>1$ we have $s_1,s_2\in \covectors$ with $s_1|_{E-A}=s_2|_{E-A}$ and $\{ s_1(y),s_2(y)\}=\{\pm 1\}$
for all $y\in A.$ For notational convenience,  assume that $s_1(y)=-1$ and $s_2(y)=+1$ for all $y\in A.$ By definition of $\covectors$ there exist $r_1,r_2\in K$ with
$r_i(x)\cdot s_i(x)\ge 0$ for $i=1,2$ and all $x\in E-A$ but $r_1(y)\le 0\le r_2(y)$ for all $y\in A.$ Any geodesic $\gamma$ connecting $r_1$ and $r_2$ in $K$ contains
a point $r$ with $r(z)=0$ for some $z\in A$ such that all interior points on the subgeodesic of $\gamma$ connecting $r_1$ and $r$ have negative coordinates at $A.$
Necessarily $r(x)\cdot s_i(x)\ge 0$ for $i=1,2$ and all $x\in E-A.$ Therefore the neighbor $s$ of $s_1$ on a shortest path between $s_1$ and $s_2$ in the graphic
hypercube ${\HH}^{(1)}(E)$ satisfying $s(z)=+1$ and $s|_{E-\{ z\}}=s_1|_{E-\{ z\}}$ belongs to $\Sign(r)\subseteq \covectors.$ A trivial induction on $\#A$ thus shows
that $\covectors$ is isometric, and thus connected. Then by \cite[Theorem 4]{BaChDrKo} we conclude that $\covectors=L(K)$ is ample.
\end{proof}

\section{Circuits and cocircuits}\label{circuit-cocircuit}

In this section we will characterize the ample sets $\covectors$ via the set $\Cocirc(\covectors)$ of the barycenter maps of the facets (maximal faces) of their cubihedra $|\covectors|$. They correspond to the set  $\Min(\Baryc(\covectors))$ of minimal elements in $\Baryc(\covectors)$ relative to the order $\prec$ on $\{ \pm 1, 0\}$ (see Definition \ref{def:upper-closure}): $\Cocirc(\covectors)=\Min(\Baryc(\covectors))$.
The members of $\Cocirc(\covectors)$ are referred to as the {\it cocircuits} of $\covectors$, since their definition bears resemblance with the one of cocircuits in  oriented matroids \cite{BjLVStWhZi}. 
Then
\begin{align*} \Baryc(\covectors)=\uparrow\hspace*{-0.1cm}\Cocirc(\covectors)&=\{ t\in \{ \pm 1,0\}^E: ~\uparrow\hspace*{-0.1cm}\{ t\}\cap \{ \pm 1\}^E\subseteq \covectors\}\\
&=\{ t\in \{ \pm 1,0\}^E: ~t|_{E-E(t)}\in \covectors^{E(t)}\}\\
&=\{ t\in \{ \pm 1,0\}^E: ~t|_{E-E(t)}\notin (\covectors^*)_{E(t)}\},
\end{align*}
where ${\covectors}^*=\{ \pm 1\}^E-\covectors$ and $E(t)=\{ x\in E: ~-1<t(x)<+1\}=t^{-1}(\{ 0\})$ for $t\in \{ \pm 1,0\}$ were defined previously.

The set $\Circ(\covectors)$ of {\it circuits} of ${\covectors}$ is
defined to be the set $\Cocirc(\covectors^*)$ of cocircuits of ${\covectors}^*$.
Clearly, $t\in \{ \pm 1,0\}^E$ is contained in $\Cocirc(\covectors^*)$
if and only if $t|_{E-E(t)}\notin {\covectors}_{E(t)}$ holds, that is, if and only
if, for every $s\in \covectors$, there exists some $x\in E$ with
$t(x)\cdot s(x)=-1$ or -- equivalently -- if and only if $s\in
\{ \pm 1\}^E$ and $t\prec s$ implies $s\notin \covectors.$ So, $\Circ\,({\covectors})$ consists of the minimal elements in $\{ \pm 1,0\} ^E$ with that
property. It is also easy to see that $\covectors$ coincides with the set of all
sign maps $s\in \{ \pm 1\}^E$ with $t\prec s$ for some $t\in \Cocirc({\covectors})$
as well as with the set of all sign maps $s\in \{ \pm 1\}^E$ with
$t\nsucc s$ for all $t\in \Circ(\covectors)$. We then obtain our final result
essentially as a corollary to Theorem \ref{barycenters-ample}:

\begin{theorem} \label{cocircuit} The following statements are equivalent for a set $\covectors\subseteq \{ \pm 1\}^E$:
\begin{itemize}
\item[\rm(i)] ${\covectors}$ is ample;
\item[\rm(ii)] $\Baryc({\covectors})$ satisfies (SCA);
\item[\rm(iii)] $\Cocirc({\covectors})$ satisfies (SCA);
\item[\rm(iv)] $\Baryc({\covectors}^*)$ satisfies (SCA);
\item[\rm(v)] $\Circ({\covectors})$ satisfies (SCA).
\end{itemize}
Furthermore, if $\hcovectors\subset \{ \pm 1,0\}^E$ satisfies (SCA) and $\Min(\hcovectors)=\hcovectors$, then the set $\covectors=\uparrow\hspace*{-0.1cm}\hcovectors\cap \{ \pm 1\}^E$ is ample and $\Cocirc(\covectors)=\hcovectors$.
\end{theorem}

\begin{proof} Given a set $\covectors\subseteq \{ \pm 1\}^E$, the associated subset $\hcovectors:=\mbox{Baryc}(\covectors)$ is upward closed, i.e.,
$\uparrow\hspace*{-0.1cm}\hcovectors=\covectors$, and yields $\covectors$ back as $\hcovectors\cap \{ \pm 1\}^E$. In particular, $[\hcovectors]=|\covectors|$ holds by definition of the two cubihedra.
Therefore, if $\covectors$ is ample, then condition (vi) of Theorem \ref{barycenters-ample} is satisfied. This establishes (i)$\Rightarrow$(ii) (or (iii), respectively).
Trivially, $\uparrow\hspace*{-0.1cm}\mbox{Cocirc}(\covectors)=\covectors,$ whence (ii)$\Longleftrightarrow$(iii) immediately follows from the equivalence of (iii) and (iv) in Theorem \ref{barycenters-ample}. If $\covectors$ satisfies (SCA), then $\covectors$ is ample by the implication from (ii) to (vi) in Theorem \ref{barycenters-ample}. Summarizing, we have shown that the first three statements
(i),(ii),(iii) are equivalent. Since $\mbox{Circ}(\covectors)=\mbox{Cocirc}(\covectors^*)$ and $\covectors$ is ample exactly when its complement $\covectors^*$ is
(cf. \cite[Theorem 2]{BaChDrKo}), statements (iv) and (v) are also equivalent to (i).  If $\hcovectors\subset \{ \pm 1,0\}$ satisfies (SCA) and $\Min(\hcovectors)=\hcovectors$, then $\covectors=\uparrow\hspace*{-0.1cm}\hcovectors\cap \{ \pm 1\}^E$ is ample and $\Baryc(\covectors)=\uparrow\hspace*{-0.1cm}\hcovectors$ by Theorem  \ref{barycenters-ample}. Since $\Min(\uparrow\hspace*{-0.1cm}\hcovectors)=\Min(\hcovectors)$, we also have $\Cocirc(\covectors)=\hcovectors$.
\end{proof}

\noindent
{\bf Remark 2.}  It follows that $r\in \mbox{Circ}({\covectors})$ for some
ample subset  ${\covectors}$ of $\{ \pm 1\}^E$ implies
${\covectors}|_A=\{ \pm 1\}^{A}$ for every proper subset $A$ of $E-E(r)$ and
hence
$${\X}({\covectors}_{E(r)})={\mathcal P}(E-E(r))-\{ E-E(r)\}.$$
In other words, for every circuit $r\in \mbox{Circ}({\covectors})$,
the support $E-E(r)$ is a ``circuit'' of ${\X} ({\covectors}),$
that is, a minimal subset of $E$ not contained in ${\X}
({\covectors}),$ while  $r|_{E-E(r)}$ is the unique element in
$\{ \pm 1\}^{E-E(r)}$ not contained in
${\covectors}_{E(r)}$. In particular, we have
$\# \mbox{Circ}({\covectors})=\#\mbox{Circ}({\X} ({\covectors}))$ with
\begin{align*}
\mbox{Circ}({\X} ({\covectors})):=\{ A\in {\mathcal P}(E)-{\X} ({\covectors}): ~B\in {\X}({\covectors})\mbox{ for all } B\subsetneqq A\}
\end{align*}
for every ample subset ${\covectors}$ of $\{ \pm 1\}^E$.

\medskip\noindent
{\bf Remark 3.} Although proved differently (in a truly combinatorial way), the equivalence between conditions (i), (iii), and (v) of Theorem \ref{cocircuit} was one of our first characterizations of ampleness. Later we found the second proof, which is more geometric and is presented here. However, one can decrypt (SCA) in the formulation of Theorem 5 of \cite{La}, characterizing the route systems of ample/lopsided sets. The route system correspond to the collection of maximal cubes of $\covectors^*$, i.e., to $\Circ({\covectors})$ in our notation. Therefore \cite[Theorem 5]{La} establishes the equivalence (i)$\Longleftrightarrow$(v) of Theorem \ref{cocircuit} (and the proof is different). 

\subsection*{Acknowledgement} V. Chepoi is partially supported by the ANR project MIMETIQUE ``Mineurs métriques'' (ANR-25-CE48-4089-01). He also would like to acknowledge J. Chalopin and K. Knauer for discussions about  \cite[Theorem 5]{La}. J.H. Koolen is partially supported by  the National Natural Science Foundation of China (No. 12471335), and the Anhui Initiative in Quantum Information Technologies (No. AHY150000).

\end{document}